\documentclass[11pt]{amsart}
\usepackage{amsmath,amssymb,amscd,amsthm,amsxtra}
\usepackage{latexsym}
\usepackage[colorlinks,citecolor=red,pagebackref,hypertexnames=false]{hyperref}

\usepackage[dvips]{graphicx}        

\newtheorem{theorem}{Theorem}[section]
\newtheorem{lemma}{Lemma}[section]
\newtheorem{proposition}{Proposition}[section]

\newtheorem{definition}{Definition}[section]
\newtheorem{corollary}{Corollary}[section]

\newcommand{\ga}{\gamma}

\newcommand{\La}{\Lambda}
\newcommand{\si}{\sigma}

\newcommand{\ti}{\tilde}

\newcommand{\wh}{\widehat}

\newcommand{\e}{\varepsilon}
\newcommand{\w}{\omega}
\newcommand{\Id}{{\bf 1}}
\newcommand{\size}{{\rm size}}
\newcommand{\ct}{{\rm count}}
\newcommand{\Tr}{{\rm Tr}}
\newcommand{\bS}{{\bf S}}
\newcommand{\bT}{{\bf T}}

\newcommand{\bP}{{\bf P}}
\newcommand{\scl}{{\rm scl}}
\newcommand{\Sh}{{\bf Sh}}
\newcommand{\card}{{\rm Card}}
\newcommand{\dSh}{{\bf dSh}}

\begin{document}

\title
[Uniform Estimates for Some Paraproducts]
{Uniform Estimates for Some Paraproducts}

\author{Xiaochun Li}

\address{
Xiaochun Li\\
Department of Mathematics\\
University of Illinois at Urbana-Champaign\\
Urbana, IL, 61801, USA}

\email{xcli@math.uiuc.edu}

\date{\today}

\thanks{Research was partially supported by the NSF}


\begin{abstract}
We establish $L^p\times L^q$ to $L^r$ estimates for some 
general paraproducts, which arise in  the study of the bilinear 
Hilbert transform along curves.
\end{abstract}
\maketitle

\section{Introduction} 

It is an important theme of current research in 
analysis to decompose 
more complicated operators, such as the Cauchy integral on Lipschitz curves 
\cite{calderon}, as a sum of simpler operators.   This theme has taken special 
prominence in multilinear Harmonic Analysis, beginning with the work 
of Lacey and Thiele \cite{LT1}, which expressed the bilinear
Hilbert 
transforms as a sum of modulated paraproducts. This theme has found 
much broader application as well.

The bilinear Hilbert transforms have a bilinear symbol 
given by restriction to a half-plane, with slope that 
depends upon the transform in question.  
In considering more complicated symbols, 
one is lead to  to paraproducts which have a complicated underlying 
description. 
One then seeks certain estimates 
of these paraproducts that are \emph{uniform} in the 
parametrizations.     This line of investigation was started in 
\cite{T}, the results of which give a new, multilinear proof 
of the boundedness of the Calderon commutator, 
fulfilling a program of study of Calderon \cite{calderon}. 
It was further extended in work of the author and 
Grafakos \cite{LL1,LL2,Li}, in the study of the disc as 
a bilinear multiplier.  Muscalu, Tao and Thiele \cite{MTT1,MTT2,MTT3}
gave alternate proofs (and more general proofs) of these results in the multilinear 
operator setting.

In this paper, we continue this line of study, 
considering certain uniform estimates that are motivated 
by an analysis of 
a blinear Hilbert transform along polynomial curves.  Namely, 
 consider the operators 
\begin{equation} \label{e.PPPPP}
(f,g) \longrightarrow  \textup{p.v.} \int _{-\infty } ^{\infty } 
f (x-y) g (x- p (y))\; \frac {dy} y \,, 
\end{equation}
for some polynomial $ p(y)$. 
The study of these operators leads to 
subtle questions in multilinear analysis, stationary phase 
methods, and paraproducts.  An initial investigation 
into operators of this type is given in \cite{FL}, where 
the polynomial is taken to be a square, and the singular 
kernel is mollified to $e^{i|t|^{-\beta}}/|t|$ for some $\beta >0$.
Without this modification, a significant difficulty might be encountered. 
There is a natural analogue of the bilinear Hilbert transform along parabolas 
in the ergodic theory setting, that is, the non-conventional ergodic 
average $\frac{1}{N}\sum_{n=0}^{N-1}f(T^nx)g(T^{n^2}x)$. In \cite{Fur}, Furstenberg
proved that the characteristic factor of the trilinear ergodic averages 
$\frac{1}{N}\sum_{n=0}^{N-1}f(T^{an})g(T^{bn})h(T^{cn})$ for all $a, b, c\in \mathbb Z$
is characteristic for the previous non-conventional ergodic average.
We are indebted to M. Lacey for bringing these Furstenberg's theorems
to our attention. 
Thus a possible method 
for the bilinear Hilbert transform along a parabola is to understand the tri-linear Hilbert 
transform first. Unfortunately, it turns out the tri-linear Hilbert transform 
is very difficult to handle. 
It is very interesting to find a proof for the bilinear Hilbert 
transform along curves without using any information of the trilinear Hilbert transform. 
It might be possible to obtain such a way by combining time-frequency analysis and 
the known results for the trilinear oscillatory integrals. This investigation 
will appear in another paper.

The paraproducts that arise have a richer parametrization 
than what has been considered before.  
The question of uniform estimates is
the main focus of this article. 
In the next section, a class of paraproducts are 
introduced.  They are  parametrized by 
\begin{itemize}
\item The  \emph{width} of the frequency window associated to the 
paraproducts, denoted by $ L_1$ and $ L_2 $ below. 
\item The \emph{overlap} of the frequency window associated 
to the paraproducts, denoted by $ M_1$ and $ M_2$ below. 
\item A \emph{modulation} of the frequency window, 
denoted by the (lower case) parameters $ n_1,n_2, 2^m$ below.  
\end{itemize}
Prior results have concentrated on the uniformity of estimates with respect to
$M_1, M_2$ from $L^p\times L^q$ to $L^r$ for $r\geq 1$ and $L_1=L_2$ \cite{MTT1}. 
The principal point of this article is to get the estimates for $1/2<r<1$ and arbitrary $L_1, L_2$.
Another new point of this article 
is the (weak) uniformity that we establish in $L_1, L_2$ and the modulation 
parameters $2^m$ (see Theorem \ref{para2est} below).  
This novelty is forced upon us by the 
stationary phase methods that one must use in the analysis of 
(\ref{e.PPPPP}).  One of anticipated applications of our theorems
is the bilinear multiplier problems associated to the symbol defined by 
a characteristic function of a suitable domain with a smooth boundary.   \\

\noindent
{\bf Acknowledgement} The author would like to thank 
his wife, Helen,  and his son, Justin, for being together through
the hard times in the past two years. And he is also very thankful to 
Michael Lacey for  his constant support and encouragement.

\section{Main Results}
\setcounter{equation}0

Let $j\in \mathbb Z$,  $L_1, L_2$ be positive integers and $M_1,
M_2$ be integers.
$$ \w_{1,j}=[2^{L_1j+M_1}/2, 2\cdot2^{L_1 j+M_1}]$$
and
$$\w_{2,j}=[-2^{L_2 j+M_2}, 2^{L_2 j+M_2}]\,.$$
Let $\Phi_1$ be a Schwartz function whose Fourier transform is  a
standard bump function supported on $[1/2, 2]$, and $\Phi_2$ be a
Schwartz function such that  $\wh\Phi_2$ is a standard bump function
supported on $[-1, 1]$ and $\wh\Phi_2(0)=1 $. For $\ell\in\{1,2\}$
and $n_1, n_2\in \mathbb Z$, define $\Phi_{\ell, j, n_\ell}$ by
$$
\wh\Phi_{\ell, j, n_\ell}(\xi)= \big(e^{2\pi i n_\ell (\cdot)} \wh\Phi_\ell (\cdot)\big) \bigg(\frac{\xi}{2^{L_\ell j +M_\ell}}\bigg)\,.
$$
It is clear that ${\wh\Phi_{\ell,j,n_\ell}}$ is supported on
$\w_{\ell,j}$. For locally integrable functions $f_\ell$'s, we
define $f_{\ell,j}$'s by
$$
f_{\ell,j, n_\ell}(x)=f_{\ell}*\Phi_{\ell, j, n_\ell}(x)\,.
$$
We define a paraproduct to be
\begin{equation}\label{defofpara0}
\Pi_{L_1, L_2, M_1, M_2, n_1, n_2}(f_1, f_2)(x) = \sum_{j\in\mathbb Z}
 \prod_{\ell =1}^2 f_{\ell,j, n_\ell}(x) \,.
\end{equation}
Another paraproduct we should introduce is the following.
For $\ell\in\{1,2\}$, let $\w'_{\ell,j}$ denote the set
$\{\xi: 2^{L_\ell j+ M_\ell}/2\leq |\xi|\leq 2\cdot 2^{L_\ell j+M_\ell}\}$.
Let $m$ be a nonnegative integer and define $\Phi_{\ell,j,m}$ by
$$
\wh\Phi_{\ell,j,m}(\xi) = \big( e^{2\pi i 2^m(\cdot)} \wh\Phi_1(\cdot)\big)\bigg(\frac{\xi}{2^{L_\ell j+ M_\ell}}\bigg)\,.
$$
Let $f_{\ell,j, m}$ be the function defined by
$$
f_{\ell,j, m}(x) =f_\ell*\Phi_{\ell,j, m}(x)\,.
$$
We define a paraproduct to be
\begin{equation}\label{type2para}
\Pi_{L_1, L_2, M_1, M_2, m}(f_1, f_2)(x)=\sum_{j\in\mathbb Z}\prod_{\ell=1}
^2 f_{\ell, j, m}(x)\,.
\end{equation}

One reason we study these paraproducts is that one will encounter 
such paraproducts in the study of the bilinear Hilbert transforms
along polynomial curves. 
We have the following uniform estimates for these paraproducts.

\begin{theorem}\label{para0uniest}
For any $p_1>1$, $p_2>1$ with $1/p_1+1/p_2=1/r$, there exists a
constant $C$ independent of $M_1, M_2, n_1, n_2$ such that
\begin{equation}\label{uniest1}
\big\|\Pi_{L_1, L_2, M_1, M_2, n_1, n_2}(f_1, f_2)\big\|_r \leq
  C\big(1+|n_1|\big)^{10}\big(1+|n_2|\big)^{10} \|f_1\|_{p_1}\|f_2\|_{p_2}\,,
\end{equation}
for all $f_1\in L^{p_1}$ and $f_2\in L^{p_2}$.
\end{theorem}

\begin{theorem}\label{para2est}
Let $\Pi_{L_1, L_2, M_1, M_2, m}(f_1, f_2)$ be the paraproduct defined 
by (\ref{type2para}). Suppose that for all $j$, 
\begin{equation}\label{2large1}
 2^{L_2 j +M_2} \geq 2^{L_1 j +M_1+m}\,.
\end{equation}
For any  $\e>0$, $p_1>1$, $p_2>1$ with $1/p_1+1/p_2=1/r$, there exists a
constant $C$ independent of $m, M_1, M_2, L_1, L_2$ such that
\begin{equation}\label{uniestpara2}
\big\|\Pi_{L_1, L_2, M_1, M_2, m}(f_1, f_2)\big\|_r \leq
  C 2^{\e m}\|f_1\|_{p_1}\|f_2\|_{p_2}\,,
\end{equation}
for all $f_1\in L^{p_1}$ and $f_2\in L^{p_2}$.
\end{theorem}

The case when $L_1=L_2$ and $r>1$ was proved in \cite{MTT1}.  
The constant $C$ in Theorem \ref{para0uniest} may depend on $L_1, L_2$. It is easy to see by the
following argument that $C$ is ${O}(\max\{2^{L_1}, 2^{L_2}\})$. It is
possible to get a much better upper bound such as $O\big(\log(1+
\max\{L_2/L_1, L_1/L_2\} )\big)$ by tracking the constants carefully in the proof we will provide. But we do not pursue the sharp constant in this article. The independence of $M_1, M_2$ is the most important issue.
In Sections \ref{para1}, \ref{para2}, we give a proof for Theorem
{\ref{para0uniest}}. The proof of Theorem \ref{para2est} will be
given in Section \ref{proofpara2}. By using Theorem \ref{para0uniest},
we get the $L^r$ bound for $\Pi_{L_1, L_2, M_1, M_2, m}$ with a operator norm
$O(2^{10m})$. Unfortunately sometimes this is not enough for our application. 
The desired norm is $O(2^{\e m})$ for a very small positive number $\e$. It might be
possible to remove the condition (\ref{2large1}) or
get the uniform estimate for $\Pi_{L_1, L_2, M_1, M_2, m}$ in which the operator norm is independent of $m$.
The uniform estimate from $L^2\times L^2$ to $L^1$ is trivial and 
(\ref{2large1}) is redundant for this case. 
In Section \ref{proofpara2}, we see that the uniform estimates 
for $\Pi_{L_1, L_2, M_1, M_2, m}$ can 
be achieved for $p_1, p_2>2$ and $1< r<2$ (see Proposition {\ref{uniestp2good}}) and (\ref{2large1}) is
superfluous for Theorem \ref{para2est} when $p_1, p_2>2$ and $1<r<2$
 (see Corollary \ref{cor91}).\\

\section{A Telescoping Argument}\label{para1}
\setcounter{equation}0

We now start to prove Theorem {\ref{para0uniest}}.
To prove Theorem {\ref{para0uniest}}, we first introduce a definition of
admissible trilinear form. And we should show that by a telescoping argument
used in \cite{LL1, T},  we can reduce the problem to estimates for an admissible trilinear form. And thus $L^r$ estimates for $r>1$
can be obtained by Littlewood-Paley theorem. The $r<1$ case is more complicated.
We have to use the time frequency analysis to deal with this case in 
Section \ref{para2}.

\begin{definition}\label{goodtri}
An admissible trilinear form is a trilinear form
\begin{equation}\label{adtridef}
\La_{L_1, L_2, M_1, M_2, n_1, n_2}(f_1, f_2, f_3) =
\int \sum_{j\in\mathbb Z}
 \prod_{\ell =1}^3 \ti f_{\ell,j, n_\ell}(x)dx\,,
\end{equation}
where $n_3=0$,
$\ti f_{\ell, j, n_\ell} = f_{\ell}*\ti\Phi_{\ell, j, n_\ell}$ and
$\ti\Phi_{\ell,j, n_\ell}$
is a function whose Fourier
transform is supported on $\ti\w_{\ell,j}$
such that
\begin{itemize}
\item[(1)] Each $\ti\w_{\ell,j}$ is an interval in $\mathbb R$ such that
  the distance from the origin to the interval is not more than $3|\ti\w_{\ell,j}|$. And $\{\ti\w_{\ell, j}\}_j$ forms a sequence of lacunary
intervals, that is, $|\ti\w_{\ell, j}|/|\ti\w_{\ell, j+1}|\leq 1/2$ for
all $j\in \mathbb Z$.
Moreover, $|\ti\w_{3,j}|\geq C\max\{|\ti\w_{1,j}|, |\ti\w_{2,j}|\}$ for some constant $C$ independent of $M_1, M_2, n_1, n_2$.
\item[(2)] There are at least two indices $\ell \in \{1,2,3\}$ such that
$\ti\Phi_{\ell, j, n_\ell}$ satisfies
\begin{equation}\label{vanish}
\wh{\ti\Phi_{\ell,j,n_\ell}}(0)=0\,\,\,
\end{equation}
\begin{equation}\label{deriest1}
\bigg | D^{\alpha}\bigg(\wh{\ti\Phi_{\ell, j, n_\ell}}\big(|\ti\w_{\ell,j}|\xi\big)\bigg)\bigg|
\leq \frac{ C_N(1+ |n_\ell|)^{\alpha}}{ (1 + |\xi|)^N}\,,
\end{equation}
for all $\xi\in\mathbb R$ and all nonnegative integers $\alpha, N$.
If an index in $\{1,2,3\}$ satisfies (\ref{vanish}) and (\ref{deriest1}),
we call the index a good index in the trilinear form $\La_{L_1, L_2, M_1, M_2, n_1, n_2}$. For the index which is not
a good index, we call it a bad index in the trilinear form
$\La_{L_1, L_2, M_1, M_2, n_1, n_2}$.
\item[(3)] If $\ell\in\{2,3\}$ is a bad index, then
$\ti\Phi_{\ell, j, n_\ell}$ satisfies (\ref{deriest1}).
Moreover, among the other two good indices $\ell'\neq \ell$, at least
one of them satisfies $|\ti\w_{\ell',j}|\leq C\min\{|\ti\w_{1,j}|,
|\ti\w_{2,j}|, |\ti\w_{3,j}|\}$ for some constant $C$ independent of $f_1$, $f_2$, $f_3$, $M_1$, $M_2$, $n_1$, $n_2$.
\item[(4)]If $1$ is a bad index, then $\ti\Phi_{1, j, n_1}$ satisfies
\begin{equation}\label{bad1}
\ti\Phi_{1, j, n_1}(x) = \sum_{k=0}^{m'(j)}\Phi_{1, j+k,
n_1}(x)\,,
\end{equation}
where $m'(j)$ is some nonnegative integer.
\end{itemize}
\end{definition}

\begin{lemma}\label{telelem1}
Let $f_3$ be a locally integrable function. Then
$$\int \Pi_{L_1, L_2, M_1, M_2, n_1, n_2}(f_1, f_2)(x)f_3(x)dx
$$
is a sum of finitely many admissible trilinear forms such that
the number of admissible trilinear forms in the sum is no more
than a constant $C$ independent of $M_1, M_2, n_1, n_2$.
\end{lemma}
\begin{proof}
For $\ell\in \{1,2\}$,  write $\w_{\ell,j}$ as $[a_{\ell,j}, b_{\ell, j}]$.
If $b_{2,j}< b_{1,j}/16$, then $|\w_{2,j}|< |\w_{1,j}|/6$ and
the distance from $\w_{1,j}+\w_{2,j}$ to the origin is not less than
$|\w_{1,j}|/4$. In this case, simply let $\ti\w_{3,j}$ be a small
neighborhood of $-(\w_{1,j}+\w_{2,j})$ and the Fourier transform of
$\ti\Phi_{3,j}$ is a suitable bump function adapted to $\ti\w_{3,j}$,
then we have the desired lemma.
Thus we now only consider the case $b_{2,j}\geq b_{1,j}/16$. Let
$\w_{3,j}$ be $[-18b_{2,j}, 18b_{2,j}]$. And $\Phi_{3,j}$ be a
Schwartz function such that its Fourier transform is a bump
function adapted to $\w_{3,j}$ and
$\wh\Phi_{3,j}(\xi)=1$ for all $\xi\in [-17b_{2,j}, 17b_{2,j}]$. Then
$$
\int\Pi(f_1,f_2)(x)f_3(x) dx = \int \sum_{j\in \mathbb Z} \prod_{\ell=1}^3
 f_{\ell,j, n_\ell}(x) dx\,,
$$
where $f_{3,j,n_3}(x)=f_3*\Phi_{3,j}(x)$ and $n_3=0$. Let $\ti\Phi_2$
be a Schwartz function such that $\wh{\ti\Phi_2}$ is a bump function on
$[-1,1] $ and $\wh{\ti\Phi_2}(\xi)=1$ for all $\xi \in[-3/4, 3/4]$. And
define $\Phi_{2,j}$ by $\wh\Phi_{2,j}(\xi)=\wh{\ti\Phi_2}(\xi/b_{2,j})$.
Let $f_{2,j}=f*\Phi_{2,j}$. We also denote $f_{3,j, n_3}$ by
$f_{3,j}$. We can replace $f_{2,j, n_2}$ by $f_{2,j}$ because
$$
\int \sum_{j\in \mathbb Z}
 f_{\ell,j, n_1}(x) \big( f_{2,j,n_2}-f_{2,j}\big)(x)f_{3,j}(x) dx\,
$$
is an admissible trilinear form. Hence the only thing we need to
show is that
$$
\La'(f_1, f_2, f_3)=\int \sum_{j\in \mathbb Z}
 f_{\ell,j, n_1}(x) f_{2,j}(x)f_{3,j} (x)dx\,
$$
is admissible. For any real number $x$, let $[x]$ denote the largest
integer not exceeding $x$. Let $m(j)$ be the integer defined by
$$
m(j) = \big[ \frac{(L_2j+M_2)-(L_1j+M_1)+6}{L_2}\big]\,.
$$
By $b_{2,j}\geq b_{1,j}/16$, we see that $m(j)\geq 0$. By a
telescoping argument, $\La'(f_1, f_2, f_3) $ equals to
$$
\int \sum_{j\in \mathbb Z}f_{1,j, n_1}(x) \sum_{k=0}^{m(j)}
 \bigg( f_{2, j-k}(x)f_{3, j-k}(x) -
  f_{2, j-k-1}(x)f_{3, j-k-1}(x)\bigg)dx\,,
$$
since $\int f_{1,j,n_1}(x)f_{2,j-m(j)-1}(x)f_{3,j-m(j)-1}(x)dx=0$
due to the following simple fact on the support of Fourier transform
of each function in the integrand, i.e.,
$$
\bigg({\rm supp}\wh f_{1,j, n_1} + {\rm supp }\wh
f_{2,j-m(j)-1}\bigg) \cap \bigg(- \big({\rm supp}\wh
f_{3,j-m(j)-1}\big)\bigg) =\emptyset\,.
$$
By a change of variables $j\rightarrow j+k$, we have that $\La'(f_1,
f_2, f_3) $ is equal to
$$
\int \sum_{j\in \mathbb Z} \sum_{k=0}^{m'(j)}f_{1,j+k, n_1}(x)
 \bigg( f_{2, j}(x)f_{3, j}(x) -
  f_{2, j-1}(x) f_{3, j-1}(x)\bigg)dx\,,
$$
where $m'(j)$ is the integer defined by
$$
m'(j) = \big[ \frac{(L_2j+M_2)-(L_1j+M_1)+6}{L_1}\big]\,.
$$
We write this integral as a sum of three parts $\La_1, \La_2, \La_3$, where
$$
\La_1=\int \sum_{j\in \mathbb Z} \bigg(\sum_{k=0}^{m'(j)}f_{1,j+k, n_1}(x)\bigg)
  f_{2, j}(x) \big( f_{3, j}(x) - f_{3, j-1}(x)\big)dx\,,
$$
$$
\La_2=\int \sum_{j\in \mathbb Z} \bigg(\sum_{k=0}^{m'(j)}f_{1,j+k,n_1}(x)\bigg)
  \big( f_{2, j}(x)-f_{2,j-1}(x) \big)
 \big(f_{3, j-1}(x) - f_{3, j-8}(x)\big)dx\,,
$$
$$
\La_3=\int \sum_{j\in \mathbb Z} \bigg(\sum_{k=0}^{m'(j)}f_{1,j+k,n_1}(x)\bigg)
  \big( f_{2, j}(x)-f_{2,j-1}(x)\big)
 f_{3, j-8}(x) dx\,.
$$
It is clear that $\La_2$ is an admissible trilinear form. Write
$\La_1$ as $\La_{11}+\La_{12}$, where
$$
\La_{11} = \int \sum_{j\in \mathbb Z} \bigg(\sum_{k=0}^{m'(j)}f_{1,j+k,n_1}(x)\bigg)
  \big(f_{2, j}(x)-f_{2,j-1}(x)\big) \big(f_{3, j}(x) - f_{3, j-1}(x)\big)dx\,,
$$
$$
\La_{12} = \int \sum_{j\in \mathbb Z} \bigg(\sum_{k=0}^{m'(j)}f_{1,j+k,n_1}(x)\bigg)
  f_{2,j-1}(x)\big(f_{3, j}(x) - f_{3, j-1}(x)\big)dx\,,
$$
Clearly, $\La_{11}$ is an admissible trilinear form. Notice that
$$
{\rm supp}\bigg( \sum_{k=0}^{m'(j)-10-[L_2/L_1]}\wh f_{1,j+k,n_1}\bigg)
\subseteq [0, 2^{-2}2^{L_2j+M_2} ] = [0, 2^{-2}2^{-L_2}b_{2,j}]\,,
$$
and
$$
{\rm supp}\big(\wh f_{3, j} - \wh f_{3, j-1}\big) \subseteq
[-18b_{2,j}, 18b_{2,j}]\backslash [-16\cdot 2^{-L_2}b_{2,j},
 16\cdot 2^{-L_2}b_{2,j}]\,.
$$
Thus $\La_{12}$ is equal to
$$
\int \sum_{j\in \mathbb Z} \bigg(\sum_{k=m'(j)-10-[L_2/L_1]}^{m'(j)}f_{1,j+k,n_1}(x)\bigg)
  f_{2,j-1, n_2}(x)\big(f_{3, j, n_3}(x) - f_{3, j-1, n_3}(x)\big)dx\,,
$$
which is obviously a finite sum of admissible trilinear forms. As
for $\La_3$, observe that
$$
{\rm supp}\bigg( \sum_{k=0}^{m'(j)-100-[L_2/L_1]}\wh
f_{1,j+k,n_1}\bigg) \subseteq [0, 2^{-80}2^{L_2j+M_2} ] = [0,
2^{-80}2^{-L_2}b_{2,j}]\,,
$$
and
$$
{\rm supp}\big(\wh f_{2, j} - \wh f_{2, j-1}\big) \subseteq
[-b_{2,j}, b_{2,j}]\backslash [-2^{-L_2-1}b_{2,j},
   2^{-L_2-1}b_{2,j}]\,.
$$
Thus $\La_{3}$ is equal to
$$
\int \sum_{j\in \mathbb Z}
\bigg(\sum_{k=m'(j)-100-[L_2/L_1]}^{m'(j)}f_{1,j+k,n_1}(x)\bigg)
  \big(f_{2,j}-f_{2,j-1}(x)\big)f_{3, j-8}(x)dx\,,
$$
which is a finite sum of admissible trilinear forms.
\end{proof}

\begin{lemma}\label{largerest}
Let $\La_{L_1, L_2, M_1, M_2, n_1, n_2}$ be an admissible trilinear form.
Then for any real numbers $p_1, p_2, p_3>1$ with $1/p_1+1/p_2+1/p_3=1$, there
exists $C$ independent of $M_1$, $M_2$, $n_1$, $n_2$ such that
\begin{equation}\label{largep}
\big| \La_{L_1, L_2, M_1, M_2, n_1, n_2}(f_1, f_2, f_3)\big|
\leq C(1+|n_1|)^{10}(1+|n_2|)^{10}\|f_1\|_{p_1}\|f_2\|_{p_2}\|f_3\|_{p_3}\,,
\end{equation}
for all $f_1\in L^{p_1}$,  $f_2\in L^{p_2}$ and $f_3\in L^{p_3}$.
\end{lemma}
\begin{proof}
If there is no bad index in the trilinear form, take $\ell_0$
to be any integer in $\{1,2,3\}$. Otherwise,
let $\ell_0$ be a bad index.  Applying Cauchy-Schwarz inequality,
$\La_{L_1, L_2, M_1, M_2, n_1, n_2}$ is dominated by
$$
\int \sup_{j\in\mathbb Z}\big|\ti f_{\ell_0, j, n_{\ell_0}}\big|
 \prod_{\ell\neq \ell_0} \bigg(\sum_j\big|\ti f_{\ell, j, n_\ell}\big|^2
\bigg)^{1/2}dx\,.
$$
Using H\"older inequality, we dominate the trilinear form by
$$
\bigg\|\sup_{j\in\mathbb Z}\big|\ti f_{\ell_0, j, n_{\ell_0}} \big|
 \bigg\|_{p_1}
 \prod_{\ell\neq \ell_0} \bigg\|\bigg(\sum_j\big|\ti f_{\ell, j, n_\ell}\big|^2
\bigg)^{1/2}\bigg\|_{p_\ell}\,.
$$
The Littlewood-Paley theorem yields that for $\ell\neq \ell_0$
$$
\bigg\|\bigg(\sum_j\big|\ti f_{\ell, j, n_\ell}\big|^2
\bigg)^{1/2}\bigg\|_{p_\ell}\leq C(1+|n_\ell|)^{10}\|f_\ell\|_{p_\ell}\,.
$$
If $\ell_0\in\{2, 3\}$, then by (\ref{deriest1}), we have
$$
 \sup_{j\in\mathbb Z}\big|\ti f_{\ell_0, j, n_{\ell_0}} \big|\leq
 (1+|n_{\ell_0}|^{10})M(f_{\ell_0})\,,
$$
which clearly yields the lemma. We now only need to consider the case $\ell_0=1$. It suffices to prove that
\begin{equation}\label{1bad}
\bigg\|\sup_j \big|\sum_{k=0}^{m'(j)} f_1*\Phi_{1, j+k, n_1}\big|\bigg\|_{p_1} \leq C(1 + |n_1|^{10})\|f_1\|_{p_1}\,.
\end{equation}
Notice that $\w_{1,j}$'s are essentailly disjoint intervals and
Fourier transform of $\sum_{k=0}^{m'(j)}\Phi_{1, j+k, n_1}$  is
supported on a bounded interval depending on $j$. The left hand side
of (\ref{1bad}) is less than
$$
 C \big\|M\big(\sum_{j}f_1*\Phi_{1, j, n_1}\big)\big\|_{p_1}\,.
$$
It is easy to verify that $\sum_{j} f_1*\Phi_{1,j, n_1}$ is a
bounded operator on $L^2$  associated to a standard Calder\'on-Zygmund
kernel by paying at most a cost of $(1+|n_1|^{10})$ in the corresponding
estimates. Thus by a standard Calder\'on-Zygmund argument, we have
for any real number $p>1$,
there is a constant $C$ independent of $M_1, M_2, n_1, n_2$ such that
$$
\big\|\sum_{j} f*\Phi_{1,j, n_1}\big\|_p \leq C(1+|n_1|^{10})\|f\|_p\,
$$
holds for all $f\in L^p$, which yields (\ref{1bad}).
 Therefore we complete the proof of the lemma.
\end{proof}

Combining Lemma {\ref{telelem1}} and Lemma {\ref{largerest}}, we
obtain (\ref{uniest1}) for $p_1, p_2, r>1$. To finish the proof of
Theorem {\ref{para0uniest}}, we need to provide a proof of $L^r$
estimate with $1/2<r\leq 1$ for  (\ref{uniest1}), which will be
given in Section {\ref{para2}}.

\section{ Time Frequency Analysis}\label{para2}
\setcounter{equation}0

In this section we prove (\ref{uniest1}) with $1/2<r\leq 1$
 for the paraproducts by time frequency analysis, which was used for establishing 
$L^p$ (uniform) estimates for the bilinear Hilbert transforms in
\cite{LL2, LT1, LT2, Li, MTT2, MTT1,  MTT3, T}.

Let $F$ be a measurable set in $\mathbb R$. $X(F)$ denotes the set
of all measurable functions supported on $F$ such that the
$L^\infty$ norms of the functions are no more than $1$. A function
in $X(F)$ can be considered essentially as the characteristic
function $\Id_F$.


To obtain Theorem {\ref{para0uniest}}, by Lemma {\ref{largerest}},
an interpolation argument in \cite{MTT2},  and the scaling
invariance, it is sufficient to prove that for any $p_1, p_2>1$ such
that $1/p_1+1/p_2 \geq 1$
 and any measurable set $F_3\subseteq\mathbb R$ with $|F_3|=1$,
there exists a subset $F'_3\subset F_3$ such that $|F'_3|\geq 1/2$
and
\begin{equation}\label{triest000}
\bigg|\int\Pi_{L_1, L_2, M_1, M_2, n_1, n_2}(f_1, f_2)(x)f_3(x)
dx\bigg|\leq
 C (1+|n_1|)^{10}(1+|n_2|)^{10}|F_1|^{1/p_1}|F_2|^{1/p_2}\,
\end{equation}
holds for all $f_1\in X(F_1), f_2\in X(F_2), f_3\in X(F'_3)$, where
$C$ is a constant independent of $f_1, f_2, f_3$, $M_1, M_2, n_1,
n_2$.

If $2^{L_2j+M_2}< 2^{L_1j+M_1}/8$,  let $\w_{3,j}=[-19\cdot
2^{L_1j+M_1}/8, -2^{L_1j+M_1}/8]$ and $\Phi_{3, j}$ be a Schwartz
function whose Fourier transform is  a bump function adapted to
$\w_{3,j}$ such that $\wh\Phi_{3,j}(\xi)=1$ for all $\xi\in [-9\cdot
2^{L_1j+M_1}/4, -2^{L_1j+M_1}/4]$. If $2^{L_2j+M_2}\geq
2^{L_1j+M_1}/8$, let $\w_{3,j}=[-18\cdot 2^{L_2j+M_2}, 18\cdot
2^{L_2j+M_2}]$ and $\Phi_{3, j}$ be a Schwartz function whose
Fourier transform is  a bump function adapted to $\w_{3,j}$ such
that $\wh\Phi_{3,j}(\xi)=1$ for all $\xi\in [-17\cdot 2^{L_2j+M_2},
17\cdot 2^{L_2j+M_2}]$. Let $n_3=0$, $\Phi_{3,j,n_3}=\Phi_{3,j} $,
$f_{3, j, n_3}(x)=f_3*\Phi_{3,j, n_3}(x)$. Define a trilinear form
$\La_{L_1, L_2, M_1, M_2, n_1, n_2}$ by
\begin{equation}\label{defoftriform}
\La_{L_1, L_2, M_1, M_2, n_1, n_2}(f_1, f_2, f_3)
 =\int \sum_{j\in\mathbb Z}\prod_{\ell=1}^3 f_{\ell, j, n_\ell}(x) dx\,.
\end{equation}
Clearly $ \La_{L_1, L_2, M_1, M_2, n_1, n_2} = \int \Pi_{L_1, L_2,
M_1, M_2, n_1, n_2}(f_1, f_2)(x)f_3(x)dx$.
Thus to prove
(\ref{triest000}), it suffices to prove the following lemma.

\begin{lemma}\label{triest0}
Let $p_1, p_2>1$ such that $1/p_1+1/p_2\geq 1 $ and $\La_{L_1, L_2,
M_1, M_2, n_1, n_2}$ be the trilinear form defined by
(\ref{defoftriform}). Let $F_1, F_2, F_3$ be measurable sets in
$\mathbb R$ with $|F_3|=1$.
Then there exists a subset $F_3'\subseteq F_3$ such that
$|F_3'|>1/2$ and there exists a constant $C$ independent of  $F_1$,
$F_2$, $F_3$, $f_1$, $f_2, f_3$, $M_1$, $M_2$, $n_1, n_2$ such that
 \begin{equation}\label{triest002}
\big|\La_{L_1, L_2, M_1, M_2, n_1, n_2}(f_1, f_2, f_3)\big|\leq
 C (1+|n_1|)^{10}(1+|n_2|)^{10}|F_1|^{1/p_1}|F_2|^{1/p_2}\,
\end{equation}
holds for all $f_1\in X(F_1), f_2\in X(F_2), f_3\in X(F'_3)$.
\end{lemma}

Lemma {\ref{triest0}} and Lemma {\ref{largerest}} implies the
estimates (\ref{uniest1}) by an interpolation argument in
\cite{MTT2}. Therefore we obtain Theorem \ref{para0uniest} once we
finish a proof of Lemma \ref{triest0}. The following subsections are
devoted to proof of Lemma \ref{triest0}.

\subsection{Definitions}
To prove Lemma \ref{triest0}, we introduce some definitions first.
Let $\psi$ be a nonnegative Schwartz function such that $\wh \psi$
is supported in $[-1/100,1/100]$ and satisfies $\wh\psi(0)=1$. Let
$\psi_k(x)=2^k\psi(2^k x)$ for any $k\in \mathbb Z$. For
$j\in\mathbb Z$ and $\ell\in\{1,2,3\}$, define  $k_{j\ell}$ to be an
integer such that $|\w_{\ell, j}|\sim 2^{k_{j\ell}}$. Denote
$\min_{\ell\in\{1,2,3\}}k_{j\ell}$ by $k_j$.
 And define
$$
I_{k_j,n} = [2^{-k_j}n,
2^{-k_j}(n+1)]\,.
$$
Define
$$
\Id^*_{j,n}(x) = \Id_{I_{k_j, n}}*\psi_{k_{j}}(x)\,.
$$
It is easy to see that
$$
\La_{L_1, L_2, M_1, M_2, n_1, n_2}(f_1, f_2, f_3) =
\int\sum_{j\in\mathbb Z}\sum_{n\in\mathbb Z}\Id^*_{j,n}(x)
\prod_{\ell=1}^3 f_{\ell,j, n_\ell}(x)dx\,.
$$
For an integer $ \gamma$ with $0\leq \ga < 2^{100}$, let $\mathbb
Z(\ga)$ be the set of all integers congruent to $\ga$ modulo
$2^{100}$.  For
${\bf S}\subset \mathbb Z(\ga)\times \mathbb Z$
 we define
\begin{equation}\label{ladef3}
\La_{\bf S}(f_1, f_2, f_3) = \int_{\mathbb R}
\sum_{(j,n)\in \mathbb {\bf S}}\Id^*_{j,n}(x)
 \prod_{\ell =1}^3  f_{\ell, j, n_\ell}(x)dx\,.
\end{equation}
$\La_{\bf S}$ depends on $L_1, L_2, M_1, M_1, n_1, n_2$. We suppress
this dependence for notational convenience. Note that there are
finite congruence classes modulo $2^{100}$. We will therefore
concentrate on proving Lemma {\ref{triest0}} for the trilinear form
$\La_{\bS}$.

In time-frequency space, each function $f_{\ell, j, n}$ for
$\ell\in\{1,2, 3\}$ corresponds to a box $I_{k_j,n}\times
 \w_{\ell, j}$. The most difficult situation is when only one of
boxes is the Heisenberg box, i.e., $|I_{k,j,n}||\w_{\ell,j}|\sim 1$.
In this situation, we can use the John-Nirenberg type argument to
get the equivalence of $L^p$ estimates of  Littlewood-Paley type 
 square functions  for only one of
functions. For other two functions, there is no such an equivalence and
an extra cost for it has to been paid if one estimates the $BMO$
norm. It turns out that the $L^p$ equivalence for at least one of
three functions is the most crucial key to solve the problem. Our
proof will heavily rely on this equivalence for one of functions.

Let $p$ be a positive number close to $1$. To obtain the Lemma
\ref{triest0}, it suffices to prove (\ref{triest002}) for $p_1\geq
p$, $p_2\geq p$ and $1/p_1+1/p_2\geq 1$. For simplicity, we only
deal with the case $n_1=n_2=n_3=0$. The general case can be handled
in the same way by paying at most a cost of
$(1+|n_1|)^{10}(1+|n_2|)^{10}$ in the constants.

We now start to prove that for $n_1=n_2=0$, any $1<p<2$ and any
measurable set $F_3$ with $|F_3|=1$ in $\mathbb R$, there exists a
subset $F'_3$ of $F_3$ with $|F_3'|\geq 1/2$ such that
\begin{equation}\label{triestp}
\big|\La_{\bS}(f_1, f_2, f_3)\big|\leq
 C |F_1|^{1/p_1}|F_2|^{1/p_2}\,
\end{equation}
holds for all $p_1\geq p, p_2\geq p$ with $1/p_1+1/p_2\geq 1$,
$f_1\in X(F_1), f_2\in X(F_2), f_3\in X(F'_3)$, where the constant
$C$ is independent of $\bS$, $F_1$, $F_2$, $F_3$, $f_1$, $f_2, f_3$,
$M_1$, $M_2$. Let us introduce some definitions first.

\begin{definition}\label{defofeset}
Let $p>1$.  Define the exceptional set  $\Omega$ by
\begin{equation}\label{defOmega}
\Omega = \bigcup_{\ell=1}^3\big\{ x\in\mathbb R:
M_p\big(M\Id_{F_\ell}\big)(x)> C_0 |F_\ell|^{1/p} \big\}
\end{equation}
where $Mf$ is the Hardy-Littlewood maximal function of $f$ and
$M_pf$ equals to $\big(M(|f|^p)\big)^{1/p}$.
\end{definition}

By this definition, for the measurable set $F_3$ with  $|F_3|=1$, we
take $F'_3=F_3\backslash \Omega$. If $C_0$ is chosen sufficiently
large we see that $|F'_3|\geq |F_3|/2$.

\begin{definition}\label{def4}
Given ${\bS}\subset \mathbb Z(\ga)\times \mathbb Z$ and
$s=(j,n)\in {\bf S}$. Let
$k_s=\min_{\ell\in\{1,2,3\}}\{k_{j\ell}\}$. The dyadic interval $
[2^{-k_s}n, 2^{-k_s}(n+1)]$ is called the time interval of $s$.  We
denote it by $I_{s}$.
\end{definition}

\begin{definition}
Let $\bS $ be a subset of $\mathbb Z(\ga)\times\mathbb Z$. We
say that $\bS$ is a convex set in $\mathbb Z(\ga)\times\mathbb
Z$ if for any $s\in \mathbb Z(\ga)\times\mathbb Z$ with
$I_{s_1}\subseteq I_s \subseteq I_{s_2}$ for some $s_1, s_2\in\bS$,
we have $s\in \bS$.
\end{definition}

\begin{definition}
Let $\bT\subset \bS$.
If there is $t\in\bT$ such that $I_{s}\subset I_{t}$ holds
for all $s\in T$, then $\bT$ is called a tree with top $t$. $\bT$ is called a
maximal tree with top $t$ in ${\bf S}$ if there does not
exist a larger tree in $\bS$ with the same top strictly containing
$\bT$.
\end{definition}

\begin{definition}\label{shadow}
Let $\bT $ be a tree in $\bS$. Define
$\scl(\bT)$ the set of scale indices
of $\bT$ by
$$
\scl(\bT)=\{j\in\mathbb Z: \exists n\in\mathbb Z, {\rm s.}\,\,{\rm t.}\,\,
  (j,n)\in \bT\}\,.
$$
For $j\in\scl(\bT)$, the $j$-th shadow of $\bT$ is defined by
$$
\Sh_j(\bT)=\bigcup\big \{ I_s: s=(j,n)\in\bT\big\} \,.
$$
Define an approximation of $\Id_{\Sh_j(\bT)}$ by
$$
\Id^*_{\Sh_j(\bT)}(x) = \Id_{\Sh_j(\bT)}*\psi_{k_j}(x)\,.
$$
\end{definition}

\begin{definition}\label{defzetanorm}
Let $(j,n)=s\in\bS$ and $\ell\in\{1,2,3\}$. And let
$$
 \Id^{**}_{ j,n}(x)=\int_{I_{k_j,n}}\frac{2^{ k_{j}}}
   {(1+ 2^{2 k_{j}}|x-y|^2)^{200}} dy
\,.
$$
 Define a semi-norm $\|f_\ell\|_{j,n}$ by
\begin{equation}\label{defseminorm}
\big\|f_\ell\big\|_{j,n}\, =\big\|f_\ell\big\|_s\,=
\frac{1}{|I_s|^{1/p}}\big\|\Id^{**}_{j,n}f_{\ell, j,
n_\ell}\big\|_p+
\frac{1}{|I_s|^{1/p}}\big\|2^{-k_{j\ell}}\Id^{**}_{j,n}Df_{\ell, j,
n_\ell}\big\|_p\,
\end{equation}
where $Df_{\ell, j, n_\ell}$ is the derivative of $f_{\ell, j,
n_\ell} $.

Define $\zeta(j,M,K)$ by
\begin{equation}\label{defofzeta}
\zeta(j,M,K)=\big[\frac{L_1j+M_1-M_2-6}{L_2}\big]+
\big[\frac{L_1}{L_2}\big]M + K\,, \end{equation}
where $L=2^{100}$,
$K$ is an integer between $-10L$ and $10L$ and $M$ is an integer
between $0$ and $6L$. For $\ell\in\{2,3\}$, we define a $\zeta$
semi-norm $\big\|f_\ell\big\|_{j,n,\zeta}$ by
\begin{equation}\label{defseminormzeta}
\big\|f_\ell\big\|_{j,n,\zeta}= \|f_\ell\|_{j,n}+\sup_{M,K}
\frac{1}{|I_s|^{1/p}}\big(\big\|\Id^{**}_{j,n}f_{\ell, \zeta(j,M,K),
0}\big\|_p+\big\||I_s|\Id^{**}_{j,n}Df_{\ell, \zeta(j,M, K),
0}\big\|_p\big)\,.
\end{equation}
For $\ell=1$, let the $\zeta$ semi-norm $\big\|f_1\big\|_{j,n,\zeta}
=\big\|f_1\big\|_{j,n}$.
\end{definition}

\begin{definition}
Let $\bT\subset \bS$ be a tree and  $t=(j_\bT, n_{\bT})\in\bT$ be
the top of $\bT$. Denote by $I_{\bT}$ the time interval of the top of
tree $\bT$.

\begin{itemize}
 \item[(a)]
  In the case $|\w_{2,j}| \leq |\w_{1,j}|/6$ for all $j\in\scl(\bT)$, define
$\Delta^*_\ell(\bT)$ for $\ell\in\{1,3\}$ by
\begin{equation}\label{defdelta1}
\Delta^*_\ell(\bT)(x) = \bigg(\sum_{(j,n)\in \bT}
\big|\Id^{**}_{j,n} f_{\ell,j,n_\ell}(x)\big|^2 \bigg)^{1/2}\,.
\end{equation}
For $\ell=2$,  define
\begin{equation}\label{defofdelta12}
\Delta^*_2(\bT)(x) = \big |\Id^{**}_{j_\bT,n_\bT}
f_{2,j_\bT,n_\ell}(x)\big|
 \,.
\end{equation}
And in this case,  for $\ell\in\{1,2,3\}$, define the $\ell$-size of
$\bT$ by
\begin{equation}\label{defofsizec1}
\size_\ell(\bT)= \frac{1}{|I_{\bT}|^{1/p}}
\big\|\Delta^*_\ell(\bT)\big\|_{p} +
\big\|f_\ell\big\|_{j_\bT,n_\bT}\,.
\end{equation}
\\

\item[(b)] In the case $|\w_{2,j}| > |\w_{1,j}|/6$ for
all $j\in\scl(\bT)$, for $\ell=2,3$,
let $f_{\ell,j, \bT}=f_{\ell,j,0}$ if $j\in\scl(\bT)$ and
$f_{\ell,j,\bT}\equiv 0$ if $j\notin \scl(\bT)$.
 Define the $\Delta^*_\ell(\bT)$ to be
\begin{equation}\label{defdelta2}
 \bigg(\sum_{(j,n)\in \bT} \big|\Id^{**}_{j,n} \big(f_{\ell,j,\bT}
- f_{\ell, j-L ,\bT} \big)(x)\big|^2 \bigg)^{1/2} +
\bigg(\sum_{(j,n)\in \bT} \big|\Id^{**}_{j,n} \big(f_{\ell,j,n_\ell}
- f_{\ell, j ,0} \big)(x)\big|^2 \bigg)^{1/2} \,.
\end{equation}
And define $\Delta^*_1(\bT)$ by
\begin{equation}\label{defdelta21}
\Delta^*_1(\bT)(x)= \bigg(\sum_{(j,n)\in \bT}
\big|\Id^{**}_{j,n}f_{1,j,n_1}(x)\big|^2 \bigg)^{1/2}\,.
\end{equation}
In this case,  for $\ell\in\{1,2,3\}$, define the $\ell$-size of
$\bT$ by
\begin{equation}\label{defofsize2b}
\size_\ell(\bT)= \frac{1}{|I_{\bT}|^{1/p}}
\big\|\Delta^*_\ell(\bT)\big\|_{p} +
\big\|f_\ell\big\|_{j_\bT,n_\bT,\zeta}\,.
\end{equation}
\end{itemize}

Let ${\bf P}$ be a subset of ${\bf S}$. Define the $\ell$-${size}^*$
 of ${\bf P}$
by
\begin{equation}\label{defofsize1}
\size^*_\ell({\bf P}) = \sup_{ \bT: \bT\subset {\bf P}}\size_\ell(\bT)\,,
\end{equation}
where $\bT$ ranges over all trees in ${\bf P}$.

In the definition of $\Id^{**}_{j,n}$, we can replace the exponent
$200$ by a larger number $2^{100}$ to define a new function. We denote
this function by  $\ti\Id^*_{j,n}$.
If $\Id^{**}_{j,n}$ is replaced by $\ti\Id^*_{j,n}$ in the definition
of $\Delta^*_\ell(\bT)$, we denote the corresponding function by
$\Delta_\ell(\bT)$.
\end{definition}

\begin{definition}\label{defcount}
Let ${\bS}$ be a subset of $\mathbb Z(\ga)\times\mathbb Z$.
Suppose that ${\bf S}$ is a union of trees $\bT\in {\mathcal F}$.
Define $\ct({\bf S})$ by
\begin{equation}\label{defofct}
 \ct({\bf S}) =\sum_{\bT\in {\mathcal F}}|I_{\bT}|\,.
\end{equation}
\end{definition}

\subsection{Reduction}
Let $\bS$ be a subset of $\mathbb Z(\ga)\times \mathbb Z$. For
$\Omega$ defined in (\ref{defOmega}), we define
\begin{equation}\label{defofbsomega}
\bS(\Omega) =\{s\in\bS: I_s\nsubseteq
\Omega\}\,.
\end{equation}
The following lemma indicates that we only need to seek the upper
bound for the trilinear form $\La_{\bS(\Omega)}$.
\begin{lemma}\label{lemomega}
Let $n_1=n_2=0$ and $f_3\in X({F_3'})$.   For all functions $f_1\in
X(F_1)$, $f_2\in X(F_2)$, the following inequality holds.
\begin{equation}\label{triest112}
\big| \La_{\bS}(f_1, f_2, f_3)-\La_{\bS(\Omega)}(f_1, f_2, f_3)\big|
\leq C\min\big\{1, |F_1|^{1/p}\big\}\min\big\{1,
|F_2|^{1/p}\big\}\,,
\end{equation}
where $C$ is a constant independent of $\bS$, $F_1$, $F_2$, $F_3$, $
f_1, f_2$, $f_3, M_1, M_2$.
\end{lemma}
\begin{proof}
Notice that if $s=(j, n)\in \bS(\Omega)^c$, then $I_s\subseteq \Omega$.
Let $\bS_{L}(\Omega)$ be defined  by
$$
\bS_L(\Omega) = \{s\in \bS(\Omega)^c: 2^LI_s\subseteq\Omega,
\,\, {\rm but}\,\, 2^{L+1}I_s\nsubseteq  \Omega  \}\,.
$$
We see that $\bS(\Omega)^c =\cup_{L=0}^\infty \bS_L(\Omega)$.
Let ${\mathcal J}_L$ be the set of all time intervals $I_s$'s for
$s\in \bS_L(\Omega)$. It is easy to see that ${\mathcal J}_L$ is
a collection of disjoint intervals and
$\sum_{J\in {\mathcal J}_L}|J|\leq |\Omega|< 1$. Hence,
it suffices to show that for any $J\in {\mathcal J}_L$ and
any $(j,n)=s\in \bS_L(\Omega)$ such that $I_s=J$, we have
\begin{equation}\label{diffest2}
\bigg| \int\Id^*_{j,n}(x)\prod_{\ell}
   f_{\ell, j, n_\ell} (x) dx\bigg|
\leq  {C 2^{-L}\min\big\{1, |F_1|^{1/p}\big\}\min\big\{1,
|F_2|^{1/p}\big\}} |J| \,,
\end{equation}
where $C$ is a constant independent of $f_1, f_2, f_3, M_1, M_2$,
since (\ref{triest112}) follows by summing all $L$'s and $J$'s
together.

We now prove (\ref{diffest2}).
Since $F'_3=F_3\backslash \Omega$ and $f_3\in X({F'_3})$, we get
for any $(j,n)\in \bS$ and any positive integer $N$,
\begin{equation}\label{estf3}
\big| \Id^*_{j,n}(x)f_{3,j,n_\ell}(x)\big|\leq
\frac{C_N}
{  \big(1+ 2^{k_{j}}{\rm dist}(x, I_s)\big)^{3N}  \big(1+ 2^{k_{j3}}{\rm dist}(x, \Omega^c)\big)^{3N}}\,.
\end{equation}
Clearly we have for $\ell\in\{1,2\}$ and $(j,n)\in\bS$,
\begin{equation}\label{estfl1}
\big| f_{\ell,j, n_\ell}(x)\big|\leq
       \int\!\frac{C_N|f_\ell(y)| 2^{k_{j\ell}}\, dy}
          { \big(1+ 2^{k_{j\ell}} |x-y| \big)^N}             \,.
\end{equation}
By the definition of $\Omega$, we have
for $\ell\in\{1,2\}$ and $(j,n)\in\bS$,
\begin{equation}\label{estfl2}
\big| f_{\ell, j, n_\ell}(x)\big|\leq {C_N \min\big\{1,
|F_\ell|^{1/p}\big\}\big(1+ 2^{k_{j\ell}}{\rm dist}(x, \Omega^c)
\big)^2}    \,.
\end{equation}
Thus (\ref{estf3}), (\ref{estfl2}) and the fact
$2^{k_{j3}}\sim 2^{\max\{k_{j\ell}\}}$ yield that
the left hand side of (\ref{diffest2}) is no more than
$$C_N 2^{-LN} \prod_{\ell=1}^2\min\big\{1, |F_\ell|^{1/p}\big\}    |J| $$
for any positive integer $N\geq 2$, which is the desired estimate.
\end{proof}

Hence, to prove (\ref{triestp}), we only need to prove the following
lemma for $\La_{\bS(\Omega)}$. The details of the proof of Lemma
\ref{ftriest} will be given in the next few subsections.


\begin{lemma}\label{ftriest}
Let $n_1=n_2=0$, $1<p <2$, $F_3\subset \mathbb R $, and
$\bS(\Omega)$ be the set defined in (\ref{defofbsomega}) and
$F'_3=F\backslash\Omega$. For all $p_1, p_2\geq p$ with
$1/p_1+1/p_2\geq 1 $, and all functions $f_1\in X(F_1)$, $f_2\in
X(F_2)$, $f_3\in X(F'_3)$, the following inequality holds.
\begin{equation}\label{triest1121}
\big|\La_{\bS(\Omega)}(f_1, f_2, f_3)\big| \leq
C|F_1|^{1/{p_1}}|F_2|^{1/{p_2}}\,,
\end{equation}
where $C$ is a constant independent of $\bS$, $F_1$, $F_2$, $F_3$,
$f_1, f_2$, $f_3, M_1, M_2$.
\end{lemma}

\subsection{Principle Lemmas}

We now state some lemmata which will be used in proof of Lemma
\ref{ftriest}.

\begin{lemma}\label{Lpest}
Let  $1<q<\infty$, $\ell\in\{1,2,3\}$ and $\bT$ be a tree in $\bS$.
Then
\begin{equation}\label{Lpest1}
\big\|\Delta^*_\ell(\bT)\big\|_{q}\leq C \inf_{x\in
I_\bT}M_q(Mf_\ell)(x)|I_\bT|^{1/q}\,,
\end{equation}
\begin{equation}\label{Lpest2}
\size_\ell(\bT)\leq C \inf_{x\in I_\bT}M_p(Mf_\ell)(x)\,,
\end{equation}
where $C$ is a constant independent of $f_\ell, \bT$, $\bS$, $M_1,
M_2$.
\end{lemma}
\begin{proof}
(\ref{Lpest1}) is a consequence of the following $L^q$ estimates of
$\Delta_\ell(\bT)$.
\begin{equation}\label{Lpest12}
\big\|\Delta^*_\ell(\bT)\big\|_q\leq C\|f_\ell\|_q\,.
\end{equation}
In fact, one can decompose $f_\ell$ into $f_\ell\Id_{2I_\bT}$ and
$f_\ell\Id_{(2I_\bT)^c}$. For the first function, apply
(\ref{Lpest12}) to get the desired estimates. For the second
function, the desired estimates follow by the fast decay due to
$\Delta^*_\ell(\bT)$ is essentially supported on $I_\bT$.

Note that we consider only the case $n_\ell=0$. For $n_\ell\neq 0$,
the following argument still works if one changes the constant $C$
to $C(1+|n_\ell|)^{5}$.  We only give the details for the case
$|\w_{2,j}|\leq |\w_{1,j}|/2$ and $\ell\in\{1,3\}$ since other cases
can be done in the same way. In this case, we have
$$
\Delta^*_\ell(\bT)(x)=\bigg(\sum_{(j,n)\in\bT}\big|\Id_{j,n}^{**}f_{\ell,
j, 0}(x)\big|^2\bigg)^{1/2}\,.
$$
Notice that $\Delta^*_\ell(\bT)(x)$ is dominated by
$$
\bigg(\sum_{j\in\mathbb Z}\big|f_{\ell, j,0}(x)\big|^2
\bigg)^{1/2}\,,
$$
where $f_{\ell, j,0}$ is defined by $\wh f_{\ell, j,0} = \wh f_\ell
\wh\Phi_{\ell, j,0}$. Note that $\wh\Phi_{\ell,j,0}$ is supported on
$\w_{\ell,j}$ and $\w_{\ell,j}$'s are disjoint. Thus the Littlewood-Paley theorem then yields the $L^q$ estimates (\ref{Lpest12}). To get
(\ref{Lpest1}), it suffices to show that
$$
\big\|\Delta^*_{\ell, {\rm out}}(\bT)\big\|_{q}\leq C \inf_{x\in
I_\bT}M_q(Mf_\ell)(x)|I_\bT|^{1/q}\,,
$$
where $\Delta^*_{\ell,{\rm out}}(\bT)$ is defined by
 $$ \Delta^*_{\ell, {\rm
 out}}(\bT)(x)=\bigg(\sum_{(j,n)\in\bT}\big|\Id_{j,n}^{**}(x)
 \big((f\Id_{(2I_\bT)^c})*\Phi_{\ell,
j, 0}\big)(x)\big|^2\bigg)^{1/2}\,.$$ By the definition of
$\Id^{**}_{j, n}$ and $\Phi_{\ell, j,0}$, we have that for any
positive integer $N$,
$$
\big|\Id_{j,n}^{**}(x)
 \big((f\Id_{(2I_\bT)^c})*\Phi_{\ell,
j, 0}\big)(x)\big| \leq \frac{C_N}{\big(1+ 2^{k_j}{\rm dist}(x,
I_s)\big)^{100}} \int_{(2I_{\bT})^c}
\frac{|f_\ell(y)|2^{k_{j\ell}}}{\big(1+2^{k_{j\ell}}|x-y|\big)^N}dy\,.
$$
which is clearly dominated by
$$
\frac{CMf_\ell(x)}{\big(1+ 2^{k_j}{\rm dist}(x, I_s)\big)^{50}
\big(1+ 2^{k_j}{\rm dist}(I_s, (2I_\bT)^c)\big)^{50} }\,.
$$
Thus for $s\in \bT$,
$$
\big\|\Id_{j,n}^{**} \big((f\Id_{(2I_\bT)^c})*\Phi_{\ell, j, 0}\big)
\big\|_q^q\leq \frac{C|I_s|}{\big(1+ 2^{k_j}{\rm dist}(I_s,
(2I_\bT)^c)\big)^{25q}} \big(\inf_{x\in
I_\bT}M_q(Mf_\ell)(x)\big)^q\,.
$$
By triangle inequality, we obtain that
$$
\big\|\Delta^*_{\ell, {\rm out}}(\bT)\big\|_q\leq \sum_{s\in \bT}
\frac{C|I_s|^{1/q}}{\big(1+ |I_s|^{-1}{\rm dist}(I_s,
(2I_\bT)^c)\big)^{25}} \inf_{x\in
I_\bT}M_q(Mf_\ell)(x)\,,
$$
which yields the desired estimate (\ref{Lpest1}). Notice that
$$
\big\| \Id^{**}_{j_\bT,n_\bT}f_{\ell,j_\bT,n_\ell} \big\|_p
 +
\big\| {2^{-k_{j_\bT\ell}}}\Id^{**}_{j_\bT,n_\bT}
 Df_{\ell,j_\bT,n_\ell} \big\|_p
 \leq \bigg\| \frac{CMf_\ell(\cdot)}{\big(1+ |I_\bT|^{-1} {\rm dist}(\cdot, I_\bT)\big)^N}
 \bigg\|_p\,,
$$
which is clearly dominated by $\inf_{x\in
 I_\bT}M_p(Mf_\ell)(x)|I_\bT|^{1/p} $. Therefore we obtain
 (\ref{Lpest2}).
\end{proof}

\begin{lemma}\label{goodbmo1}
Suppose that  $s=(j,n)\in\bS$.

If $2^{k_{j\ell}}\sim 2^{k_j}$, then
\begin{equation}\label{smallbound}
\big\|\Id^{**}_{j,n} f_{\ell,j,n_\ell}\big\|_\infty \leq
C\big\|f_\ell\big\|_{j,n}
\end{equation}
holds for $\ell\in\{1,2,3\}$, where $C$ is a constant independent of
$s,  f_\ell, n_\ell$.

If $2^{k_{j1}}\sim 2^{k_j}$, then
\begin{equation}\label{smallboundzeta}
\big\|\Id^{**}_{j,n} f_{\ell,\zeta(j,M, K),n_\ell}\big\|_\infty \leq
C\big\|f_\ell\big\|_{j,n,\zeta}
\end{equation}
holds for $\ell\in\{2,3\}$, where $\zeta(j,M, K) $ is defined in
Definition \ref{defzetanorm} and $C$ is a constant independent of
$s,  f_\ell, n_\ell, \zeta, M, K$.
\end{lemma}
\begin{proof}
We only prove (\ref{smallbound}) since (\ref{smallboundzeta})
essentially  is a consequence of (\ref{smallbound}).
 Let $\mu=\big\|f_\ell\big\|_{j,n}$.
 By the definition of the semi-norm, we have
\begin{equation}\label{smalldelta2}
\big\|\Id^{**}_{j,n} f_{\ell,j,n_\ell}\big\|_p +
\big\|{|I_{s}|}\Id^{**}_{j,n}
 Df_{\ell,j,n_\ell} \big\|_p \leq \mu |I_s|^{1/p}\,.
\end{equation}
 First we prove the BMO estimate for the function, that is
\begin{equation}\label{smallbmo}
\big\|\Id^{**}_{j,n} f_{\ell,j,n_\ell}\big\|_{BMO} \leq C\mu\,.
\end{equation}
If $|I_s|\leq |J|$, by (\ref{smalldelta2}) we have
$$\inf_c\int_J\big| \Id^{**}_{j,n}(x)f_{\ell,j,n_\ell}(x)-c\big| dx
  \, \leq \, \big\|\Id^{**}_{j,n} f_{\ell,j,n_\ell} \big\|_p|J|^{1-\frac{1}{p}}
 \,\leq \, \mu |I_s|^{\frac{1}{p}}|J|^{1-\frac{1}{p}}
  \, \leq \, \mu |J|\,.
$$
If $|I_s|\geq |J|$, by (\ref{smalldelta2}) we obtain that
\begin{eqnarray*}
 &  & \inf_c\int_J\big| \Id^{**}_{j,n}(x)f_{\ell,j,n_\ell}(x)-c\big| dx\\
 & \leq & |J|\int_J\bigg| \big(\Id^{**}_{j,n}f_{\ell,j,n_\ell}\big)' (x)\bigg| dx\\
 & \leq & |J|\int_J \big|\big(\Id^{**}_{j,n}\big)'(x)f_{\ell,j,n_\ell}(x)  \big|dx + |J|\int_J\big| \Id^{**}_{j,n}(x)Df_{\ell,j,n_\ell}(x) \big|dx\\
 & \leq & C|J||I_s|^{-1}\big\|\Id^{**}_{j,n}f_{\ell,j,n_\ell}\big\|_p|J|^{1-\frac{1}{p}} + |J|\big\| \Id^{**}_{j,n}Df_{\ell,j,n_\ell}\big\|_p
 |J|^{1-\frac{1}{p}}\\
 &\leq & C\mu |J|^{2-\frac{1}{p}} |I_s|^{\frac{1}{p}-1}\,\leq\, C\mu|J|\,.
\end{eqnarray*}
Thus we get the BMO estimate (\ref{smallbmo}). Interpolating (\ref{smallbmo})
and (\ref{smalldelta2}), we have for any $p\leq q <\infty$,
$$
\big\|\Id^{**}_{j,n}f_{\ell,j,n_\ell}\big\|_q\leq C\mu
|I_s|^{1/q}\,.
$$
Notice that an integration by parts and H\"older inequality yield that
$$
\big\|\Id^{**}_{j,n} f_{\ell,j,n_\ell}\big\|_\infty \leq
 \big\| \Id^{**}_{j,n} f_{\ell,j,n_\ell}  \big\|_{p'}^{1/2}
 \big\| \big( \Id^{**}_{j,n} f_{\ell,j,n_\ell}  \big)'  \big\|_p^{1/2}\,,
$$
where $1/p+1/p'=1$. Hence the desired estimate (\ref{smallbound})
 follows by (\ref{smalldelta2}) and $L^{p'}$ estimates for the functions.
\end{proof}

\begin{lemma}\label{goodbmo2}
Suppose that $2^{k_{j\ell}}\sim 2^{k_j}$ holds for
all $(j,n)\in \bS$.
Then for any tree $\bT $ in $\bS$, we have
\begin{equation}\label{smallBMOp}
\big\| \Delta_\ell(\bT) \big\|_{BMO} \leq C \size^*_\ell(\bT) \,,
\end{equation}
where $C$ is a constant independent of $ \bT, \bS, L_1, L_2, M_1, M_2,
 f_\ell, n_\ell$.
\end{lemma}
\begin{proof}
We only give the a proof for $\ell=1$. Other cases can be handled in the same way. Let $\mu=\size^*_\ell(\bS)$.
Let $J$ be a dyadic interval and $\bT_J=\{s\in\bT: I_s\subseteq J\}$.
We then dominate $\inf_c\int_J\big|\Delta_\ell(\bT)(x) - c\big|dx$ by a sum of
the following three parts.
$$
  \int_J \bigg(\sum_{s\in \bT_J} \big|\ti\Id^{*}_{j,n}(x) f_{\ell,j,n_\ell}
(x) \big|^2\bigg)^{1/2} dx\,,
$$
$$
 \int_J \bigg(\sum_{\substack{s\in \bT\backslash\bT_J\\|I_s|\leq |J|}} \big|\ti\Id^{*}_{j,n}(x) f_{\ell,j,n_\ell}
(x) \big|^2\bigg)^{1/2} dx\,,
$$
and
$$
 \inf_c\int_J \bigg|\bigg(\sum_{\substack{s\in \bT\backslash\bT_J\\|I_s|>|J|}} \big|\ti\Id^{*}_{j,n}(x) f_{\ell,j,n_\ell}
(x) \big|^2\bigg)^{1/2}-c\bigg| dx\,.
$$
The first part is clearly dominated by $\mu |J|$ because of the H\"older inequality and the fact that $\mu$ is the $\ell$-${\rm size}^*$ of $\bS$.

Since $p\leq 2$ we estimate the second part by
\begin{eqnarray*}
  & &  \bigg\| \bigg(\sum_{\substack{s\in \bT\backslash\bT_J\\|I_s|\leq |J|}} \big|\ti\Id^{*}_{j,n} f_{\ell,j,n_\ell}\big|^2\bigg)^{1/2}   \bigg\|_{L^p(J)}|J|^{1-\frac{1}{p}}\\
 & \leq &  \bigg(\sum_{\substack{s\in \bT\backslash\bT_J\\|I_s|\leq |J|}} \big\|\ti\Id^{*}_{j,n} f_{\ell,j,n_\ell}\big\|^p_{L^p(J)}\bigg)^{1/p}
|J|^{1-\frac{1}{p}}\\
  & \leq   & \bigg(\sum_{\substack{s\in \bT\backslash\bT_J\\|I_s|\leq |J|}}
 \frac{ C
 \big\|\Id^{**}_{j,n} f_{\ell,j,n_\ell}\big\|^p_p  }
 {\big( 1+ |I_s|^{-1}{\rm dist}(J, I_s)\big)^{100}}\bigg)^{1/p}|J|^{1-\frac{1}{p}}\\
 &\leq & \mu \bigg(\sum_{\substack{s\in \bT\backslash\bT_J\\|I_s|\leq |J|}}
 \frac{ C |I_s| }
 {\big( 1+ |I_s|^{-1}{\rm dist}(J, I_s)\big)^{100}}   \bigg)^{1/p} |J|^{1-\frac{1}{p}}\,\, \leq \,\, C\mu|J|\,.
\end{eqnarray*}

The third part is estimated by
\begin{eqnarray*}
  &  &  \bigg(\inf_c\int_J \bigg|\bigg(\sum_{\substack{s\in \bT\backslash\bT_J\\|I_s|>|J|}} \big|\ti\Id^{*}_{j,n}(x) f_{\ell,j,n_\ell}
(x) \big|^2\bigg)^{1/2}-c\bigg|^2dx\bigg)^{1/2}|J|^{1/2}\\
  & \leq &\bigg(\inf_c\int_J \bigg|\sum_{\substack{s\in \bT\backslash\bT_J\\|I_s|>|J|}} \big|\ti\Id^{*}_{j,n}(x) f_{\ell,j,n_\ell}
(x) \big|^2-c\bigg|dx\bigg)^{1/2}|J|^{1/2}\\
  & \leq & C\bigg(\int_J \sum_{\substack{s\in \bT\backslash\bT_J\\|I_s|>|J|}} \bigg|
  \bigg(\big|\ti\Id^{*}_{j,n}(x) f_{\ell,j,n_\ell}
(x) \big|^2\bigg)'\bigg|dx  \bigg)^{1/2}|J|\,,
\end{eqnarray*}
which is dominated by a sum of the following two terms,
$$
R_1= C\bigg(\int_J \sum_{\substack{s\in
\bT\backslash\bT_J\\|I_s|>|J|}}
  |I_s|^{-1}\big|\ti\Id^{*}_{j,n}(x) f_{\ell,j,n_\ell}
(x) \big|^2 dx  \bigg)^{1/2}|J|\,,
$$
and
$$
R_2= C\bigg(\int_J \sum_{\substack{s\in
\bT\backslash\bT_J\\|I_s|>|J|}}\big|\ti\Id^{*}_{j,n}(x)f_{\ell,j,n_\ell}
(x) \big| \big|\ti\Id^{*}_{j,n}(x) Df_{\ell,j,n_\ell} (x) \big| dx
\bigg)^{1/2}|J|\,,
$$
By Lemma \ref{goodbmo1}, we see that for any
$q\geq p$,
$$
 \big\|\Id^{**}_{j,n}f_{\ell,j,n_\ell}\big\|_q\leq
 C\mu|I_\bT|^{1/q}\,.
$$
Thus, by H\"older inequality, the first term $R_1$ is estimated by
\begin{eqnarray*}
 &  & C\bigg( \sum_{\substack{s\in
\bT\backslash\bT_J\\|I_s|>|J|}}\frac{|I_s|^{-1}}{\big(1+|I_s|^{-1}{\rm
dist}(J, I_s)\big)^{100}}
\big\|\Id^{**}_{j,n}f_{\ell,j,n_\ell}\big\|_4^2 |J|^{1/2}
\bigg)^{1/2}|J| \\
 & \leq & C\mu\bigg( \sum_{\substack{s\in
\bT\backslash\bT_J\\|I_s|>|J|}}\frac{|I_s|^{-1/2}|J|^{1/2}}{\big(1+|I_s|^{-1}{\rm
dist}(J, I_s)\big)^{100}}\bigg)^{1/2}|J| \,\,\leq \, C\mu |J|\,,
\end{eqnarray*}
and the second term $R_2$ is estimated by
\begin{eqnarray*}
  &  &C\bigg(\sum_{\substack{s\in
\bT\backslash\bT_J\\|I_s|>|J|}}\big\|\ti\Id^{*}_{j,n}f_{\ell,j,n_\ell}
 \big\|_{L^{p'}(J)} \big\|\Id^{**}_{j,n}Df_{\ell,j,n_\ell}
 \big\|_p\bigg)^{1/2}|J|\\
 & \leq &C\bigg( \mu \sum_{\substack{s\in
\bT\backslash\bT_J\\|I_s|>|J|}}\frac{|I_s|^{\frac{1}{p}-1}
\big\|\Id^{**}_{j,n}f_{\ell,j,n_\ell}
 \big\|_{p'+1}|J|^{\frac{1}{p'(p'+1)}} }
{\big(1+|I_s|^{-1}{\rm dist}(J, I_s)\big)^{100}}
 \bigg)^{1/2}|J|\\
 & \leq & C\mu\bigg( \sum_{\substack{s\in
\bT\backslash\bT_J\\|I_s|>|J|}}\frac{|I_s|^{-\frac{1}{p'(p'+1)}}
|J|^{\frac{1}{p'(p'+1)}} } {\big(1+|I_s|^{-1}{\rm dist}(J,
I_s)\big)^{100}} \bigg)^{1/2}|J|\,\,\leq \,\, C\mu|J|\,.
\end{eqnarray*}
This completes the proof of (\ref{smallBMOp}).
\end{proof}

The principal lemma is the following organization lemma.
\begin{lemma}\label{prilem}
Let $\ell\in\{1,2, 3\}$ and ${\bf S}$ be a subset of $\mathbb
Z(\ga)\times \mathbb Z$. ${\bf S}$ can be partitioned to two
parts ${\bf S}_1$ and ${\bf S}_2$ such that ${\bf S}_1$ is a union
of maximal trees with
\begin{equation}\label{s1est}
\ct({\bf S}_1) \leq C\big(\size^*_\ell({\bf S})\big)^{-p}|F_\ell|\,,
\end{equation}
and
\begin{equation}\label{s2est}
\size^*_\ell({\bf S}_2)\leq \frac{1}{2}\size^*_\ell({\bf S})\,,
\end{equation}
where $C$ is a constant independent of $\bS, M_1, M_2, f_\ell,
F_\ell$.
\end{lemma}
\begin{proof}
Let ${\mathcal F}_0$ be the set of all trees $\bT\subset \bS$ such
that $\size_\ell(\bT) > \size^*_\ell(\bS)/2$. Recall that $I_\bT$ is
the time interval for the top of $\bT$.  Let $\mathcal I$ denote the
collection of all possible $I_\bT$'s for trees $\bT\in\mathcal F_0$.
Initially, set $\bS_1:=\emptyset$, $\mathcal I_{\rm stock}:=\mathcal
I$, and $\bS_{\rm stock}:=\bS$. Take a longest interval $J$ in
$\mathcal I_{\rm stock}$. By the defintion of $\mathcal I$, there
must be a tree $\bT\in \mathcal F_0$ whose top is $J$. Let $\ti\bT$
be the maximal tree in $\bS_{\rm stock}$ with the top $J$. Obviously
$\size_\ell(\ti\bT)\geq \size^*_\ell(\bS)/2$.
 We remove
this maximal tree from $\bS_{\rm stock}$.  Update $\bS_{\rm stock}:=
\bS_{\rm stock}\backslash \ti\bT$,  $\bS_1 :=\bS_1\cup\ti\bT$,
and  
$$\mathcal I_{\rm stock}:= \mathcal I_{\rm
stock}\backslash\{I\in\mathcal I_{\rm stock}: I\subseteq J\} \,.$$
Repeat this procedure until $\mathcal I_{\rm stock}=\emptyset$.
Clearly when this process terminates, $\bS_1$ is a union of a trees
$\ti\bT$'s and $I_{\ti\bT}$'s are disjoint due to the maximality of
trees. By (\ref{Lpest2}) and the size condition on $\ti\bT$, we have
$$
\inf_{x\in I_{\ti\bT}}M_p(Mf_\ell)(x)\geq \size^*_\ell(\bS)/2\,,
$$
which implies that
$$
\bigcup_{\ti\bT} I_{\ti\bT} \subseteq \big\{x\in \mathbb R:
 M_p(Mf_\ell)(x)\geq \size^*_\ell(\bS)/2\big\}\,.
$$
Thus the disjointness property of $I_{\ti\bT}$'s and  (weak)
$L^q$ estimates for $1\leq q\leq \infty $ of Hardy-Littlewood
maximal functions yield (\ref{s1est}).
Let $\bS_2=\bS\backslash \bS_1$. Clearly $\bS_2$ satisfies
(\ref{s2est}). Therefore we complete the proof of Lemma
{\ref{prilem}}.
\end{proof}

\subsection{The size estimate for a tree}

Let $\bS$ be a convex subset
of $\mathbb Z(\ga)\times \mathbb Z$. By the definition of
$\bS(\Omega)$ in (\ref{defofbsomega}), it is clear that $\bS(\Omega)$
is convex. Partition $\bS(\Omega)$ into two subsets $\bS^{(1)}(\Omega)$ and $\bS^{(2)}(\Omega)$, where
\begin{equation}\label{defofSgood}
\bS^{(1)}(\Omega) = \big\{(j,n)\in\bS(\Omega): |\w_{2,j}|\leq
 |\w_{1,j}|/6\big\}
\end{equation}
\begin{equation}\label{defofSbad}
\bS^{\rm (2)}(\Omega) = \big\{(j,n)\in\bS(\Omega): |\w_{2,j}|>
 |\w_{1,j}|/6\big\}\,.
\end{equation}
For any $(j,n)\in\bS^{(1)}(\Omega)$, $k_{j2}=k_{j}$ by the definition
of $k_j$. And for any $(j,n)\in\bS^{(2)}(\Omega)$,
$ 2^{k_{j1}}\sim 2^{k_j}$.

\begin{lemma}\label{convexitylem}
For $\kappa\in\{1, 2\}$, $\bS^{(\kappa)}(\Omega)$ is convex.
\end{lemma}
\begin{proof}

We only prove the lemma for $\kappa=2$.
One can prove the lemma for $\kappa=1$ similarly.
Let $s_1=(j_1, n_1), s_2=(j_2, n_2)$ in $\bS^{(2)}(\Omega)$. And
$s=(j,n)\in \mathbb Z(\ga)\times\mathbb Z$ such that
$ I_{s_2}\subseteq I_s\subseteq I_{s_1}$. By the convexity of $\bS(\Omega)$
we get $s\in\bS(\Omega)$. In order to get $s\in\bS^{(2)}(\Omega)$, we
need to show that $|\w_{2,j}| > |\w_{1,j}|/6$.
The simple case is the case $2^{k_j}=|\w_{1,j}|$. In this case, 
 $|\w_{1,j_2}|/10 \leq |\w_{1,j}|\leq  10|\w_{1, j_1}|$, which implies $j_2\leq j\leq j_1$. Since $|\w_{2, j_1}|>|\w_{1, j_1}|/6$
and $|\w_{2,j_2}|>|\w_{1, j_2}|/6$, the linearity of
the function $f(j)=(L_1j+M_1)-(L_2j+M_2)$
yields that $|\w_{2,j}|>|\w_{1, j}|/6$.

We now turn to another case $2^{k_j}=|\w_{2,j}|$. Since $I_s$
is nested between $I_{s_1}$ and $I_{s_2}$, we get
$|\w_{1, j_2}|/10 \leq |\w_{2,j}|\leq 10|\w_{1, j_1}|$.
The first half part of this inequality and the definition of $k_{j}$ imply
$j_2\leq  j$. And the second half part of the inequality and the fact
$(j_1, n_1)\in\bS^{(2)}(\Omega)$ yield $j\leq j_1$. Thus we
get $|\w_{2,j}|>|\w_{1,j}|/6$ by the linearity of the function $f(j)$.
 Hence $s$ must be in $\bS^{(2)}(\Omega)$ in
either case.  This proves the lemma.
 \end{proof}

\begin{lemma}\label{shadowest}
Let $\kappa\in\{1,2\}$, $\bT$ be a convex tree in $\bS^{(\kappa)}(\Omega)$  with the top $t=(j_\bT, n_\bT)$ and
$\partial\Sh_j(\bT)$ be the boundary of the $j$-th shadow of $\bT$.
Let $\card (\partial\Sh_j(\bT))$ denote the cardinality of
the boundary of the $j$-th shadow. Then
\begin{equation}\label{shadowest1}
\sum_{j\geq j_\bT}2^{-k_j}\card(\partial\Sh_j(\bT)) \leq C|I_\bT|\,,
\end{equation}
where $C$ is a constant independent of $\bT$.
\end{lemma}
\begin{proof}
This lemma is similar to one technical lemma (Lemma 4.8) in \cite{MTT3}. We give a
similar proof. 
Note that the $j$-th shadow consists of finite disjoint intervals and
its boundary thus contains all endpoints of the intervals. It is
sufficient to consider only all left endpoints since the right endpoints
can be handled in the same way. Let $\partial_{\rm left}(\Sh_j(\bT))$ denote
the collection of all left endpoints of the intervals in the $j$-th shadow.
Let $z\in \partial_{\rm left}(\Sh_j(\bT))$ and
$I_j(z)=(z-2^{-k_j}, z-2^{-k_j}/2)$.
To prove (\ref{shadowest1}), it suffices to show that the
intervals $I_j(z)$'s are disjoint for all possible $j, z$.
Assume that there are $j, j'\in\scl(\bT)$, $z\in
\partial_{\rm left}(\Sh_j(\bT))$ and $z'\in \partial_{\rm left}(\Sh_{j'}(\bT))$   such that $(j, z)\neq (j', z')$ and $I_j(z)\cap I_{j'}(z')\neq \emptyset$.
By the nesting property of dyadic intervals and the fact that
$z-2^{-k_j}$ is an endpoint of some dyadic intervals, we see that
$j\neq j'$. Without loss of generality, suppose that $j<j'$. The
fact that $I_j(z)$ and $I_{j'}(z')$ have nonempty intersection then
implies $z'\in (z-2^{-k_j}, z)$. Since $z$ is a left endpoint of
some intervals in the $j$-th shadow, $z'$ can not be in
$\Sh_j(\bT)$. However, the convexity of $\bT$ yields that
$\Sh_{j'}(\bT)\subseteq \Sh_j(\bT)$. This is a contradiction.
Therefore
 we obtain the lemma.
\end{proof}

\begin{lemma}\label{difftree}
Let $\kappa\in\{1,2\}$,
$\bT$ be a convex tree in $\bS^{(\kappa)}(\Omega)$ and
$\ti\La_\bT(f_1, f_2,  f_3)$ be defined by
\begin{equation}\label{deftila}
\ti\La_\bT(f_1, f_2, f_3)=\sum_j\int\prod_{\ell=1}^3
\sum_{n\in\bT_j} F_{\ell, j, n}(x)dx \,,
\end{equation}
where $\bT_j=\{n\in \mathbb Z: (j,n)\in \bT\}$ and
$F_{\ell, j, n}$ is defined by
\begin{equation}
 F_{\ell, j, n}(x)=\Id^*_{j,n}(x)f_{\ell, j, n_\ell}(x)\,.
\end{equation}
Then we have
\begin{equation}\label{tri1diffest}
 \big|\La_\bT(f_1, f_2, f_3)-\ti\La_\bT(f_1, f_2, f_3)\big|
\leq C\size^*_1({\bT})\size^*_2({\bT})|I_{\bT}|\,,
\end{equation}
where $C$ is a constant independent of $\bT, {\bS}, f_1, f_2, f_3$.
\end{lemma}
\begin{proof}
Observe that the difference  $|\La_\bT-\ti\La_\bT|$ by
$$
\sum_{j\in\scl(\bT)}\int \bigg|\Id^*_{\Sh_j(\bT)}(x) -
 \big(\Id^*_{\Sh_j(\bT)}\big)^3 (x) \bigg|\prod_{\ell=1}^3 \big|f_{\ell,j, n_\ell}(x)\big| dx\,,
$$
which is dominated by
$$
\sum_{j\in\scl(\bT)}\sum_{I: |I|=2^{-k_j}}
\int_I \bigg|\bigg(\Id^*_{\Sh_j(\bT)}(x) -
 \big(\Id^*_{\Sh_j(\bT)}\big)^3 (x)\bigg)
 \big(\ti\Id^{*}_{\Sh_j(\bT)}(x)\big)^{-\frac{1}{10}} \bigg|
  \Pi_{j,\bT}(f_1, f_2, f_3)(x)dx\,,
$$
where
\begin{equation}\label{defoftish}
\ti\Id^{*}_{\Sh_j(\bT)}(x) = \int_{\Sh_j(\bT)} \frac{2^{k_j}}{\big(1+2^{2k_j}|x-y|^2\big)^{2^{1000}}} dx \,
\end{equation}
and
$$
\Pi_{j,\bT}(f_1, f_2, f_3)(x)=
\prod_{\ell=1}^3
  \big|\big(\ti\Id^{*}_{\Sh_j(\bT)}\big)^{1/30} f_{\ell,j, n_\ell}(x)\big| \,.
$$
H\"older inequality, Lemma \ref{goodbmo1} and (\ref{Lpest1}) then yield
that
\begin{equation}\label{sigle}
\big\|\Pi_{j, \bT}(f_1, f_2, f_3)\big\|_{L^1(I)}
\leq C\size^*_1({\bT})\size^*_2({\bT})
 2^{-k_j} \,.
\end{equation}
Thus we estimate the difference $|\La_\bT-\ti\La_\bT|$ by
$$
C\size^*_1({\bT})\size^*_2({\bT})\sum_{j\in\scl(\bT)}\sum_{I: |I|=2^{-k_j}}
|I| \bigg\|\bigg(\Id^*_{\Sh_j(\bT)} -
 \big(\Id^*_{\Sh_j(\bT)}\big)^3\bigg)
 \big(\ti\Id^{*}_{\Sh_j(\bT)}\big)^{-1/10} \bigg\|_{L^\infty(I)}\,,
$$
By the definition of $\Id^{*}_{\Sh_j(\bT)}$, it is easy to see that
it is a smooth approximation of $\Id_{\Sh_j(\bT)}$ and for any
positive interger $N$ the following inequality holds.
$$
|I|\bigg\|\bigg(\Id^*_{\Sh_j(\bT)} -
 \big(\Id^*_{\Sh_j(\bT)}\big)^3\bigg)
 \big(\ti\Id^{*}_{\Sh_j(\bT)}\big)^{-1/10} \bigg\|_{L^\infty(I)}
\leq \frac{C_N|I|}{\big( 1+|I|^{-1}{\rm dist}(I, \partial\Sh_j(\bT))\big)^N}\,.
$$
Summing up all $I$'s with $|I|=2^{-k_j}$, we estimate the difference by
$$
C\size^*_1({\bT})\size^*_2({\bT})\sum_{j\in \scl(\bT)} 2^{-k_j}\card(\partial\Sh_j(\bT)) \,.
$$
Hence the lemma follows by Lemma \ref{shadowest}.
\end{proof}

\begin{lemma}\label{diffest23}
Let $\bT$ be a convex tree in $\bS^{(2)}(\Omega)$.
For $\ell\in\{2,3\}$,
let $F_{\ell, j}$ be defined by
\begin{equation}\label{defofFlj}
 F_{\ell,j}(x)=\Id^*_{\Sh_j(\bT)}(x)f_{\ell,j, 0}(x)\,,
\end{equation}
if $\bT_j\neq \emptyset$, and $F_{\ell, j}\equiv 0$ if
$\bT_j=\emptyset$.
Then we have
\begin{equation}\label{diffestfinal}
\sup_M\bigg\|\bigg(\sum_j\big|
F_{\ell,j-M}-F_{\ell,j-M-L}\big|^2\bigg)^{1/2}\bigg\|_p\leq
C\size^*_\ell(\bT)|I_\bT|^{1/p}\,,
\end{equation}
where $L=2^{100}$, $M$ ranges over all integers between $0$ and $6L$
and $C$ is a constant independent of $f_\ell, \bT$.
\end{lemma}
\begin{proof}
For simplicity, we only prove the lemma for $M=0$.
 It is easy to see that
$|F_{\ell, j}-F_{\ell,j-L}(x)|$ is dominated by
$$
\big|\Id^*_{\Sh_j(\bT)}(x)\big(f_{\ell, j, 0}(x)- f_{\ell,j-L,0}(x)\big)  \big| +
\big|\big(\Id^*_{\Sh_j(\bT)}(x) - \Id^*_{\Sh_{j-L}(\bT)}(x)\big)f_{\ell,j-L}(x)\big|\,.
$$
Clearly,  by the definition of $\Delta^*_\ell(\bT)$ and $\size_\ell^*(\bT)$,
we get
$$
\bigg\|\bigg(\sum_j\big|\Id^*_{\Sh_j(\bT)}\big(f_{\ell, j, 0}- f_{\ell,j-L,0}\big) \big|^2\bigg)^{1/2}\bigg\|_p \leq
 C\big\|\Delta^*_\ell(\bT)\big\|_p\leq C\size^*(\bT)|I_\bT|^{1/p}\,.
$$
Thus to obtain (\ref{diffestfinal}), it suffices to show that
\begin{equation}\label{diffestfinal1}
\bigg\|\bigg(\sum_j\big|\big(\Id^*_{\Sh_j(\bT)} - \Id^*_{\Sh_{j-L}(\bT)}\big)f_{\ell,j-L,0}   \big|^2\bigg)^{1/2}\bigg\|_p\leq C\size^*(\bT)|I_\bT|^{1/p} \,.
\end{equation}
Heuristically one can consider $\Id^*_{\Sh_j(\bT)}$ as $ \Id_{\Sh_j(\bT)}$.
Then by the nesting property of the $j$-th shadows due to the convexity of
the tree, we see that $\Sh_{j-L}(\bT)\backslash\Sh_j(\bT)$'s are disjoint
and this is the reason why we have such an estimate.

Now we go to the technical details.
Since $p\leq 2$, we  estimate
the left hand side of (\ref{diffestfinal1}) by
$$
\bigg(\sum_{j\in\scl(\bT)}\big\|\big(\Id^*_{\Sh_j(\bT)} - \Id^*_{\Sh_{j-L}(\bT)}\big) f_{\ell, j-L,0}\big\|_p^p \bigg)^{1/p}\,.
$$
This is dominated by
$$
\bigg(\sum_{j\in\scl(\bT)}\sum_{I:|I|=2^{-k_j}}
\int_I \bigg| \big(\Id^*_{\Sh_j(\bT)} - \Id^*_{\Sh_{j-L}(\bT)}\big)(x)
 (\ti\Id^*_{\Sh_{j-L}(\bT)}(x))^{-\frac{1}{10}} \Pi_j^*(f_\ell)(x) \bigg|^p dx \bigg)^{1/p}\,,
$$
where $\ti\Id^*_{\Sh_j(\bT)}$ is the function defined in (\ref{defoftish})
and $\Pi^*_j(f_\ell)=(\ti\Id^*_{\Sh_{j-L}(\bT)})^{1/10}f_{\ell,j-L,0}$.
H\"older inequality, Lemma \ref{goodbmo1} and (\ref{Lpest1}) then yield
that
\begin{equation}\label{sigle2}
\big\|\Pi^*_{j}(f_\ell)\big\|_{L^p(I)}
\leq C \size^*_1({\bT})\size^*_2({\bT})|I|^{1/p} \,.
\end{equation}
Thus we dominate the left hand side of (\ref{diffestfinal1}) by
$$
C\size^*_1({\bT})\size^*_2({\bT})
\bigg( \sum_{j\in\scl(\bT)}\sum_{I:|I|=2^{-k_j}}
 \bigg\| \big(\Id^*_{\Sh_j(\bT)} - \Id^*_{\Sh_{j-L}(\bT)}\big)
 (\ti\Id^*_{\Sh_{j-L}(\bT)})^{-\frac{1}{10}}     \bigg\|_{L^\infty(I)} |I|
\bigg)^{1/p}
$$
Since $\Sh_{j}(\bT)\subset\Sh_{j-L}(\bT)$, it is easy to see that
$$
 \big| \Id^*_{\Sh_j(\bT)}(x) - \Id^*_{\Sh_{j-L}(\bT)}(x) \big|
 \leq  C \ti\Id^*_{\Sh_{j-L}(\bT)}(x)\,.
$$
On the other hand,
observe that
$|\Id^*_{\Sh_j(\bT)}-\Id^*_{\Sh_{j-L}(\bT)}| $ is dominated by
$$
 \dSh^*_j(x) = \Id_{\Sh_{j-L}(\bT)\backslash \Sh_{j}(\bT)}*\psi_{k_{j-L}}(x)
+ \frac{C_N}{\big(1+ 2^{k_j}{\rm dist}(x, \partial(\Sh_j(\bT)))\big)^N}\,,
$$
for any positive integer $N$. Hence the $L^{\infty}(I)$ norm
of
$ \big(\Id^*_{\Sh_j(\bT)} - \Id^*_{\Sh_{j-L}(\bT)}\big)
 (\ti\Id^*_{\Sh_{j-L}(\bT)})^{-\frac{1}{10}} $ is estimated
by
$$
\frac{C_N}{\big(1+|I|^{-1}{\rm dist}(I, \Sh_{j-L}(\bT)\backslash\Sh_{j}(\bT))   \big)^N } + \frac{C_N}{\big(1+|I|^{-1}{\rm dist}(I, \partial(\Sh_j(\bT))  ) \big)^N }\,.
$$
For those $I$'s contained in $\Sh_j(\bT)$, we have
$$
\frac{1}{\big(1+|I|^{-1}{\rm dist}(I, \Sh_{j-L}(\bT)\backslash\Sh_{j}(\bT))   \big)^N }\leq \frac{1}{\big(1+|I|^{-1}{\rm dist}(I, \partial(\Sh_j(\bT))  ) \big)^N }\,.
$$
For those $I$'s contained in $(\Sh_{j-L}(\bT))^c$, we get
$$
\frac{1}{\big(1+|I|^{-1}{\rm dist}(I, \Sh_{j-L}(\bT)\backslash\Sh_{j}(\bT))   \big)^N }\leq \frac{1}{\big(1+|I|^{-1}{\rm dist}(I, \partial(\Sh_{j-L}(\bT))  ) \big)^N }\,.
$$
Thus we have
\begin{eqnarray*}
  &   & \sum_{I: |I|=2^{-k_j}} \frac{1}{\big(1+|I|^{-1}{\rm dist}(I, \Sh_{j-L}(\bT)\backslash\Sh_{j}(\bT))\big)^N  }\\
 & \leq &
 |I|^{-1} \big|\Sh_{j-L}(\bT)\backslash\Sh_j(\bT)\big|
+ \card\big(\partial\Sh_j(\bT)\big) + \card\big(\partial\Sh_{j-L}(\bT)\big)\,.
\end{eqnarray*}
By the nesting property of $j$-th shadows, the fact $2^{k_j}\sim 2^{k_{j-L}}$,
and Lemma \ref{shadowest}, we obtain that
$$
\sum_{j\in\scl(\bT)}\sum_{I:|I|=2^{-k_j}}
 \bigg\| \big(\Id^*_{\Sh_j(\bT)} - \Id^*_{\Sh_{j-L}(\bT)}\big)
 (\ti\Id^*_{\Sh_{j-L}(\bT)})^{-\frac{1}{10}}     \bigg\|_{L^\infty(I)} |I|
\leq C|I_\bT|\,,
$$
which yields the desired estimate (\ref{diffestfinal1}).
Therefore we finish the proof.
\end{proof}

\begin{lemma}\label{sizelem}
Let $\kappa\in\{1, 2\}$
and $\bT$ be a convex tree in $\bS^{(\kappa)}(\Omega)$.  Then we have
\begin{equation}\label{lemest1115}
\big|\La_{\bT}(f_1, f_2, f_3)\big|
\leq  C\size^*_1({\bT})\size^*_2({\bT})|I_{\bT}|\,,
\end{equation}
where $C$ is a constant independent of $\bT, {\bS}, f_1, f_2, f_3$.
\end{lemma}
\begin{proof}
By Lemma \ref{difftree}, it is sufficient to show that
\begin{equation}\label{lemest1116}
\big|\ti\La_{\bT}(f_1, f_2, f_3)\big|
\leq  C\size^*_1({\bT})\size^*_2({\bT})|I_{\bT}|\,,
\end{equation}
where $C$ is a constant independent of $\bT, {\bS}, f_1, f_2, f_3$.\\

We first prove the simple case $\kappa=1$. In this case,
$k_{j2}= k_j$ for all $(j, n)\in\bT$.
We thus dominate $|\ti\La_\bT|$ by
$$
\int_{\mathbb R} \sup_{j}\bigg|\sum_{n\in\bT_j}F_{2,j, n}(x)\bigg|
 \prod_{\ell\neq 2} \bigg( \sum_{(j,n)\in\bT}\big| F_{\ell, j, n}(x)\big|^2 \bigg)^{1/2} dx\,.
$$
By the definition of $\Delta_\ell$ and H\"older inequality, we estimate
$|\La_\bT|$ by
$$
\big\|\sup_{(j,n)\in\bT}\big|F^*_{2,j, n_2}\big| \big\|_\infty
 \big\|\Delta_1(\bT)\big\|_p \big\|\Delta_3(\bT)\big\|_{p'}\,,
$$
where $1/p+1/p'=1$ and $F^*_{\ell,j, n}=\Id^{**}_{j,n}f_{\ell, j, n_\ell}$.
 Lemma \ref{goodbmo1} yields that
$$
\big\| F^*_{2,j, n}\big\|_\infty\leq \size^*_2({\bT})\,.
$$
Clearly the definition of size yields
$$\|\Delta_1(\bT)\|_p\leq \size^*_1({\bT})|I_\bT|^{1/p} \,.$$
And (\ref{Lpest1}) yields
$$
\big\|\Delta_3(\bT)\big\|_{p'}\leq C|I_\bT|^{1/p'}\,.
$$
Putting all of them together, we obtain (\ref{lemest1115}) for the case $\kappa=1$.
\\

We now prove the case $\kappa=2$. In this case, $ 2^{k_j}\sim
2^{k_{j1}}$ for all $(j,n)\in\bT $.
For simplicity, we only consider the case $n_\ell=0$. The general case can be done in the same way
by paying a cost of $(1+|n_\ell|)^{10}$ in the constant.
Then we write the trilinear form $\ti\La_\bT$ as
$$
\ti\La_\bT(f_1, f_2, f_3)=\sum_{j\in\mathbb Z} \int\prod_{\ell=1}^{3}
   F_{\ell, j}(x) dx\,,
$$
where $F_{\ell,j}$ is defined in (\ref{defofFlj}).  Here we take a
convenient notation that $F_{\ell, j}$ is identically zero if $j\notin
\scl(\bT)$. Let $L=2^{100}$. By the telescoping argument used in
Lemma \ref{telelem1}, we can write $\ti\La_\bT$ as a finite sum of
two types of trilinear forms. One type of them is defined by
\begin{equation}\label{goodtype}
\La_{\bT, 1}(f_1, f_2, f_3)=\int\sum_{j\in\mathbb Z } F_{1,
j+m'(j)-M}(x)
 \Pi_{j,L}(F_{2,j}, F_{3,j})(x)dx\,,
\end{equation}
where $m'(j)=[(L_2j+M_2-L_1j-M_1+6)/L_1]$, $M$ is an integer between
$0$ and $6L$, and $\Pi_{j,L}(F_{2,j}, F_{3,j})$ equals to
$(F_{2,j}-F_{2, j-L})
 F_{3, j-8L}$ or $F_{2, j-L}(F_{3,j}-F_{3,j-L})$.
Another type of them is defined by
\begin{equation}\label{badtype}
\int\sum_{j\in\mathbb Z } \bigg(\sum_{k=0}^{m'(j)}F_{1, j+k}(x)
\bigg)
 \big(F_{2,j}(x)-F_{2,j-L}(x)\big)\big(F_{3,j-M}(x)-F_{3,j-M-L}(x)\big)dx\,,
\end{equation}
which is denoted by $\La_{\bT, 2}(f_1, f_2, f_3)  $. \\

We now prove the estimate for the first type trilinear form $\La_{\bT,1}$.
Let us first consider  the case 
$$
\La_{\bT,1}(f_1, f_2, f_3)=\int\sum_{j\in\mathbb Z } F_{1,
j+m'(j)-M}(x)(F_{2,j}-F_{2, j-L})(x)
 F_{3, j-8L}(x)dx\,.
$$
In this case,  by  Cauchy-Schwarz inequality,
$|\La_{\bT,1}|$ is estimated by
$$
 \int \bigg(\sum_j\big|F_{1,j+m'(j)-M}(x)F_{3,j-8L}(x)\big|^2\bigg)^{1/2}
   \bigg(\sum_j\big| F_{2,j}(x)-F_{2,j-L}(x)\big|^2\bigg)^{1/2}
   dx\,.
$$
Using H\"older inequality,  we dominate it by
$$
\bigg\|\bigg(\sum_j\big|F_{1,j+m'(j)-M}F_{3,j-8L}\big|^2\bigg)^{1/2}
\bigg\|_{p'}\bigg\|\bigg(\sum_j\big|
F_{2,j}-F_{2,j-L}\big|^2\bigg)^{1/2} \bigg\|_p\,.
$$
The first factor in this expression is no more than
$$
\bigg\|\bigg(\sum_j\bigg|
 \sum_{n\in\bT_{j+m'(j)-M}}\Id^*_{j+m'(j)-M,n}f_{1,j+m'(j)-M,n_1}
 f_{3,j-8L,0}\bigg|^2\bigg)^{1/2} \bigg\|_{p'}\,,
$$
which is dominated by
$$
\bigg\|\bigg(\sum_j
 \sum_{n\in\bT_{j+m'(j)-M}}\bigg|\big(\ti\Id^{*}_{j+m'(j)-M,n}\big)^2f_{1,j+m'(j)-M,n_1}
 f_{3,j-8L,0}\bigg|^2\bigg)^{1/2} \bigg\|_{p'}\,.
$$
We estimate it by
$$
\bigg\|\bigg(\sum_{(j,n)\in\bT}
 \big|\ti\Id^{*}_{j,n}f_{1,j,n_1}\big|^2\bigg)^{1/2} \bigg\|_{p'}
 \sup_{(j,n)\in\bT}\big\|\ti\Id^{*}_{j,n}f_{3,\zeta(j,M,K),0}\big\|_\infty\,,
$$
where $K$ is some integer between $-10L$ and $10L$ and
$\zeta(j,M,K)$is defined as in (\ref{defofzeta}). Clearly,
$\ti\Id^*_{j,n}f_{3,\zeta(j,M,K),0}$ is bounded. Also by Lemma
\ref{goodbmo2} and an interpolation, we have
\begin{equation}\label{est1term}
\bigg\|\bigg(\sum_j\big|\ti\Id^{*}_{j,n}f_{1,j,n_1}
\big|^2\bigg)^{1/2} \bigg\|_{p'}\leq
C\size^*_1(\bT)|I_\bT|^{1/p'}\,.
\end{equation}
And Lemma \ref{diffest23} yields that
\begin{equation}\label{est2term}
\bigg\|\bigg(\sum_j\big| F_{2,j}-F_{2,j-L}\big|^2\bigg)^{1/2}
\bigg\|_p\leq \size_2^*(\bT)|I_\bT|^{1/p}\,.
\end{equation}
(\ref{est1term}) and (\ref{est2term}) give us the desired estimate
for $\La_{\bT,1}$ in the first case.

We now consider the case 
$$
\La_{\bT,1}(f_1, f_2, f_3)=\int\sum_{j\in\mathbb Z } F_{1,
j+m'(j)-M}(x)F_{2,j-L}(x) \big( F_{3, j}-F_{3, j-L}\big)(x)dx\,.
$$
In this case, using Cauchy-Schwarz inequality, we have that
$|\La_{\bT,1}|$ is estimated by
$$
 \int \bigg(\sum_j\big|F_{1,j+m'(j)-M}(x)F_{2,j-L}(x)\big|^2\bigg)^{1/2}
   \bigg(\sum_j\big| F_{3,j}(x)-F_{3,j-L}(x)\big|^2\bigg)^{1/2}
   dx\,.
$$
By H\"older inequality,  we dominate it by
$$
\bigg\|\bigg(\sum_j\big|F_{1,j+m'(j)-M}F_{2,j-L}\big|^2\bigg)^{1/2}
\bigg\|_{p'}\bigg\|\bigg(\sum_j\big|
F_{3,j}-F_{3,j-L}\big|^2\bigg)^{1/2} \bigg\|_p\,.
$$
The first factor in this expression is no more than
$$
\bigg\|\bigg(\sum_j
 \sum_{n\in\bT_{j+m'(j)-M}}\bigg|\big(\ti\Id^{*}_{j+m'(j)-M,n}\big)^2f_{1,j+m'(j)-M,n_1}
 f_{2,j-L,0}\bigg|^2\bigg)^{1/2} \bigg\|_{p'}\,.
$$
We estimate it by
$$
\bigg\|\bigg(\sum_{(j,n)\in\bT}
 \big|\ti\Id^{*}_{j,n}f_{1,j,n_1}\big|^2\bigg)^{1/2} \bigg\|_{p'}
 \sup_{(j,n)\in\bT}\big\|\ti\Id^{*}_{j,n}f_{2,\zeta(j,M,K),0}\big\|_\infty\,,
$$
where $K$ is some integer between $-10L$ and $10L$ and
$\zeta(j,M,K)$is defined as in (\ref{defofzeta}). By
(\ref{smallboundzeta}) and the definition of size, we see that
\begin{equation}\label{2case2est}
\sup_{(j,n)\in\bT}\big\|\ti\Id^{*}_{j,n}f_{2,\zeta(j,M,K),0}\big\|_\infty
\leq C\size^*_2(\bT)\,.
\end{equation}
Lemma \ref{diffest23} and (\ref{Lpest2}) yield that
\begin{equation}\label{est3term}
\bigg\|\bigg(\sum_j\big| F_{3,j}-F_{3,j-L}\big|^2\bigg)^{1/2}
\bigg\|_p\leq |I_\bT|^{1/p}\,.
\end{equation}
Putting (\ref{est1term}), (\ref{2case2est}) and (\ref{est3term})
together, we thus get the desired estimate for $\La_{\bT,1}$ in the
second case.\\

Finally let us estimate $\La_{\bT,2}$. The integrand in
(\ref{badtype}) is dominated by
$$
\sup_j\bigg|\sum_{k=0}^{m'(j)}F_{1, j+k}(x) \bigg|
\bigg(\sum_{j\in\mathbb Z }
 \big|\big(F_{2,j}-F_{2,j-L}\big)(x)\big|^2\bigg)^{\frac{1}{2}}
\bigg(\sum_{j\in\mathbb Z }
 \big|\big(F_{3,j-M}-F_{3,j-M-L}\big)(x)\big|^2\bigg)^{\frac{1}{2}}.
$$
There exist $p_1, p_3\in\mathbb R$ such that $1/p_1+1/p+1/p_3=1$ and
$p_1>p', p_3>1$. By H\"older inequality we dominate $\La_{\bT,2}$ by
$$
\bigg\|\sup_j\bigg|\sum_{k=0}^{m'(j)}F_{1, j+k}(x)
\bigg|\bigg\|_{p_1}\bigg\| \bigg(\sum_{j\in\mathbb Z }
 \big|F_{2,j}-F_{2,j-L}\big|^2\bigg)^{\frac{1}{2}}\bigg\|_{p}
\bigg\|\bigg(\sum_{j\in\mathbb Z }
 \big|F_{3,j-M}-F_{3,j-M-L}\big|^2\bigg)^{\frac{1}{2}}\bigg\|_{p_3}.
$$
Just notice that one can simply define the size with respect to any
number $p_3$ by using $L^{p_3}$, then (\ref{Lpest2}) and Lemma
\ref{diffest23} still hold. Thus we have
\begin{equation}\label{anyqest}
\bigg\|\bigg(\sum_{j\in\mathbb Z }
\big|F_{3,j-M}-F_{3,j-M-L}\big|^2\bigg)^{\frac{1}{2}}\bigg\|_{p_3}
\leq C|I_\bT|^{1/p_3}
\end{equation}
Notice that the supports of Fourier transform of $F_{1,j+k}$'s are
essentially disjoint. We thus have
$$
\bigg\|\sup_j\bigg|\sum_{k=0}^{m'(j)}F_{1, j+k}(x)
\bigg|\bigg\|_{p_1}\leq C\bigg \|\sum_j F_{1,j}\bigg\|_{p_1}\,.
$$
Clearly,
$$
 \bigg\|\sum_j F_{1,j}\bigg\|_2 \leq
 \big\|\Delta_1(\bT)\big\|_2\,.
$$
By Lemma \ref{goodbmo2} and an interpolation, we have that
$$
\big\|\Delta_1(\bT)\big\|_2\leq C\size_1^*(\bT)|I_\bT|^{1/2}\,.
$$
Thus we get
$$
\bigg\|\sum_j F_{1,j}\bigg\|_2\leq C\size_1^*(\bT)|I_\bT|^{1/2}\,.
$$
A routine argument as we did in Lemma \ref{goodbmo2} yields
\begin{equation}\label{bmoest111}
\bigg\|\sum_j F_{1,j}\bigg\|_{BMO}\leq C\size_1^*(\bT)\,.
\end{equation}
Now by an interpolation, we obtain that
\begin{equation}\label{p1normest}
\bigg\|\sum_j F_{1,j}\bigg\|_{p_1} \leq
C\size_1^*(\bT)|I_\bT|^{1/p_1}\,.
\end{equation}
Hence the desired estimate for $\La_{\bT,2}$ now follows by
(\ref{p1normest}), (\ref{est2term}) and (\ref{anyqest}). Therefore
we obtain Lemma \ref{sizelem}.

\end{proof}

\subsection{Proof of Lemma \ref{ftriest}}
We now prove Lemma \ref{ftriest}.  Without loss of generality, we
can assume that $\bS$ is a convex set.  Lemma \ref{convexitylem}
then yields that $\bS^{(1)}(\Omega)$ and $\bS^{(2)}(\Omega)$ are
convex. By the definition of convexity, we see that the convexity is
preserved for a maximal tree in a convex set and the remaining set
obtained by removing a maximal tree from a convex set. Thus,
applying the organization lemma \ref{prilem} for
$\bS^{(\kappa)}(\Omega)$ inductively, we decompose
\begin{equation}\label{goodde}
\bS^{(\kappa)}(\Omega)=\bigcup_{\sigma} \bS^{(\kappa)}_\si\,,
\end{equation}
where $\kappa \in\{1, 2\}$,  $\sigma $ ranges over all possible dyadic numbers,
$\bS^{(\kappa)}_\si=\cup_{\bT\in\mathcal F_\si^{(\kappa)}}\bT$
such that $\mathcal F_\si^{\kappa} $ is a collection of convex trees
with
\begin{equation}\label{ctest1}
 \ct(\bS_\si^{(\kappa)})\leq C\si^{-p}\,,
\end{equation}
and for both $\ell=1$ and $\ell=2$,
\begin{equation}\label{sizekaest}
\size^*_\ell(\bS_\si^{(\kappa)})\leq \si|F_\ell|^{1/p}\,.
\end{equation}
By Lemma \ref{Lpest} and the definition of $\bS(\Omega)$,
we know that $\sigma \leq 1$ in order to make
$\bS_\si^{(\kappa)}$ nonempty and
 we can also sharpen the upper bound in
the size estimate for $\bS_\si^{(\kappa)}$ by 
\begin{equation}\label{sizekaest1}
\size^*_\ell(\bS_\si^{(\kappa)})\leq \min\{1, \si|F_\ell|^{1/p}\}\,.
\end{equation}
Hence we estimate $\La_{\bS(\Omega)}$ by
$$
\big|\La_{\bS(\Omega)}(f_1, f_2, f_3)\big|\leq \sum_{\kappa=1}^2
 \sum_{\si \leq 1} \sum_{\bT\in\mathcal F_\si^{(\kappa)}}
\big|\La_\bT(f_1, f_2, f_3) \big|\,.
$$
Lemma \ref{sizelem} yields that
$$
 \big|\La_{\bS(\Omega)}(f_1, f_2, f_3)\big|\leq \sum_{\kappa=1}^2
 \sum_{\si \leq 1} \sum_{\bT\in\mathcal F_\si^{(\kappa)}}
 \size^*_1(\bS_\si^{(\kappa)})  \size^*_2(\bS_\si^{(\kappa)}) |I_\bT|       \,.
$$
Applying (\ref{sizekaest1}) and (\ref{ctest1}), we thus obtain
\begin{equation}\label{finalest}
\big|\La_{\bS(\Omega)}(f_1, f_2, f_3)\big|\leq
 C\sum_{\sigma\leq 1}\min\{1, \si|F_1|^{1/p}\}
   \min\{1, \si|F_2|^{1/p}\}\si^{-p}\,,
\end{equation}
which clearly implies (\ref{triest1121}).
Therefore we complete the proof of Lemma {\ref{ftriest}}.\\

\section{Proof of Theorem \ref{para2est}}\label{proofpara2}
\setcounter{equation}0

We now prove Theorem \ref{para2est}.
The uniform estimate from $L^2\times L^2$ to $L^1$
follows immediately by a change of variables and Littlewood-Paley theory
and (\ref{2large1}) is superfluous. 
Take this simple idea and we can get the uniform estimate for $p_1, p_2>2$
and $1<r<2$ in Proposition \ref{uniestp2good} for 
the case $2^{L_2j+M_2}< 2^{L_1j+M_1}/8$ or $2^{L_1j+M_1}<2^{L_2j+M_2}/8$.
For the general case, we pay a cost of $m$ in the operator norm
in this range of $p_1, p_2, p$ to get Lemma \ref{uniestp2}.

For $r<1$ case, we use some idea from Section \ref{para2} and one can
see that technically it is much simpler than what we did in Section \ref{para2}.
We have to assume (\ref{2large1}) and pay a little more for the operator norm such as $2^{\e m}$ (see Lemma \ref{uniestp11}). 
The uniform estimate might be true but $2^{\e m}$ for a small $\e>0$ is good enough for our application.

As we did in Section \ref{para2}, we set up a trilinear form first.
Let us ignore the condition (\ref{2large1}) for a while. 
 If $2^{L_2j+M_2}< 2^{L_1j+M_1}/8$,  let $\w'_{3,j}=\{\xi:
2^{L_1j+M_1}/8\leq |\xi|\leq 19\cdot 2^{L_1j+M_1}/8\}$ and $\Phi_{3, j}$ be a Schwartz
function whose Fourier transform is a bump function adapted to
$\w'_{3,j}$ such that $\wh\Phi_{3,j}(\xi)=1$ for all $2^{L_1j+M_1}/4  \leq
|\xi|\leq 9\cdot 2^{L_1j+M_1}/4$.

If $2^{L_1j+M_1}< 2^{L_2j+M_2}/8$,  let $\w'_{3,j}=\{\xi:
2^{L_2j+M_2}/8\leq |\xi|\leq  19\cdot 2^{L_2j+M_2}/8\}$ and $\Phi_{3, j}$ be a Schwartz
function whose Fourier transform is a bump function adapted to
$\w'_{3,j}$ such that $\wh\Phi_{3,j}(\xi)=1$ for all $2^{L_2j+M_2}/4  \leq
|\xi|\leq 9\cdot 2^{L_2j+M_2}/4$.

If $  2^{L_1j+M_1}/8 \leq 2^{L_2j+M_2}\leq
  8\cdot 2^{L_1j+M_1}$, let
$\w'_{3,j}=\{\xi: |\xi|\leq 18\cdot \max\{2^{L_1j+M_1}, 2^{L_2j+M_2}\}\}$
and $\Phi_{3, j}$ be a Schwartz function whose
Fourier transform is  a bump function adapted to $\w'_{3,j}$ such
that $\wh\Phi_{3,j}(\xi)=1$ for all $  |\xi| \leq 17\cdot
\max\{2^{L_1j+M_1}, 2^{L_2j+M_2}\}$.
Let $\Phi_{3,j,m}=\Phi_{3,j} $,
$f_{3, j, m}(x)=f_{3,j,0}(x)=f_3*\Phi_{3,j, 0}(x)$. Define a trilinear form
$\La_{L_1, L_2, M_1, M_2, m}$ by
\begin{equation}\label{defoftriformlam}
\La_{L_1, L_2, M_1, M_2, m}(f_1, f_2, f_3)
 =\int \sum_{j\in\mathbb Z}\prod_{\ell=1}^3 f_{\ell, j, m}(x) dx\,.
\end{equation}
Clearly $ \La_{L_1, L_2, M_1, M_2, m} = \int \Pi_{L_1, L_2,
M_1, M_2, m}(f_1, f_2)(x)f_3(x)dx$.

We will prove the following two lemmata.

\begin{lemma}\label{uniestp20}
Let $p_1, p_2> 2$ and $1< r < 2$ such that $1/p_1+1/p_1=1/r$. Let
$F_1, F_2, F_3$ be measurable sets in $\mathbb R$. There exists a
constant $C$ independent of $F_1, F_2, F_3, f_1, f_2, f_3$, $M_1$,
$M_2$, $m$ such that
\begin{equation}\label{largepest20}
\big|\La_{L_1, L_2, M_1, M_2,m}(f_1, f_2, f_3)\big|\leq
Cm|F_1|^{1/p_1}|F_2|^{1/p_2}|F_3|^{1/r'}\,
\end{equation}
holds for all $f_1\in X(F_1)$, $f_2\in X(F_2)$ and $f_3\in X(F_3)$.
\end{lemma}

\begin{lemma}\label{uniestp110}
Let $\e$ be any positive number, $1<p<2$ and $F_1, F_2, F_3$ be
measurable sets in $\mathbb R$ such that $|F_3|=1$. 
Suppose (\ref{2large1}) holds for all $j$'s. Then 
there is a
subset $F_3'\subset F_3$ with $|F'_3|\geq |F_3|/2$ such that for all
$p_1, p_2\geq p$ with $1/p_1+1/p_2\geq 1$, and all functions $f_1\in
X(F_1)$, $f_2\in X(F_2)$,  $f_3\in X(F_3)$, the following inequality
holds.
\begin{equation}\label{psmallest110}
\big|\La_{L_1, L_2, M_1, M_2, m}(f_1, f_2, f_3)\big|\leq C2^{\e
m}|F_1|^{1/p_1}|F_2|^{1/p_2}
 \, ,
\end{equation}
where $C$ is a constant independent of $\bS$, $F_1, F_2, F_3, f_1,
f_2, f_3, M_1, M_2, m$.
\end{lemma}

Theorem \ref{para2est} is a consequence of these two lemmas
by using interpolation and duality. We also have a corollary from 
Lemma \ref{uniestp20} by a simple interpolation.

\begin{corollary}\label{cor91}
Let $p_1, p_2> 2$ and $1< r < 2$ such that $1/p_1+1/p_1=1/r$.
There exists a
constant $C$ independent of $F_1, F_2, F_3, f_1, f_2, f_3$, $M_1$,
$M_2$, $m$ such that
\begin{equation}\label{largepest200}
\big\|\Pi_{L_1, L_2, M_1, M_2,m}(f_1, f_2)\big\|_r\leq
Cm \|f_1\|_{p_1}\|f_2\|_{p_2}\,
\end{equation}
holds for all $f_1\in L^{p_1}$ and $f_2\in L^{p_2}$.
\end{corollary}

\subsection{Proof of Lemma \ref{uniestp20}}\label{subpl}

For $\ell\in\{1,2,3\}$, let $\Tr_{\ell,j, m}$ be a translation
function defined by
\begin{equation}\label{defoftr}
\Tr_{\ell,j,m}(x)=x+{m_{j\ell}}\,,
\end{equation}
where $m_{j\ell}=2^{m-jL_\ell-M_\ell}$ if $\ell\in\{1,2\}$ and
$m_{j3}=0$. Notice that
$f_{\ell,j,m}(x)=f_{\ell,j,0}(\Tr_{\ell,j,m}(x))$. Write $\La_{L_1,
L_2, M_1, M_2, m}$ as
$$
\La_{L_1, L_2, M_1, M_2, m}(f_1, f_2, f_3)= \int_{\mathbb
R}\prod_{\ell=1}^{3}\sum_{(j, n)\in\mathbb Z\times \mathbb Z}
\Id^*_{j,n}\big(\Tr_{\ell,j,m}(x)\big)f_{\ell,j,
0}\big(\Tr_{\ell,j,m}(x)\big)
 dx\,.
$$

For $\bS\subset \mathbb Z(\gamma)\times \mathbb Z$ we define
\begin{equation}\label{deflam}
\La_{\bS, m}(f_1, f_2, f_3)=\int_{\mathbb R}\sum_{j\in\mathbb Z}
\prod_{\ell=1}^{3}\sum_{n\in\bS_j}
F_{\ell,j, n, m}(x) dx\,,
\end{equation}
where $\bS_j=\{n: (j,n)\in \bS\}$ and  $F_{\ell,j,n,m}$ is defined by
\begin{equation}\label{defFljnm}
F_{\ell, j, n, m}(x)=
\big( (\Id^*_{j,n}f_{\ell,j, 0})\circ\Tr_{\ell,j,m}\big)(x)\,.
\end{equation}

Let $k_{ j\ell}$ be an integer such that $|\w'_{\ell,j}|\sim 2^{k_{
j\ell}}$. For $s=(j,n)\in\bS$, let $k_s=k_j=\min_\ell{k_{j\ell}}$.
The time interval of $s$ is defined by $I_s=[2^{-k_s}n,
2^{-k_s}(n+1)]$. We then can define a tree in $\bS$ as in Section
\ref{para2}. To prove Lemma \ref{uniestp20}, it is sufficient to
prove the following lemma.

\begin{lemma}\label{uniestp2}
Let $p_1, p_2> 2$ and $1< r < 2$ such that $1/p_1+1/p_1=1/r$.
Let $F_1, F_2, F_3$ be measurable sets in $\mathbb R$.
There exists a constant $C$ independent of $F_1, F_2, F_3, f_1, f_2, f_3$,
$M_1$, $M_2$, $m$
such that
\begin{equation}\label{largepest2}
\big|\La_{\bS, m}(f_1, f_2, f_3)\big|\leq Cm|F_1|^{1/p_1}|F_2|^{1/p_2}|F_3|^{1/r'}\,
\end{equation}
holds for all $f_1\in X(F_1)$, $f_2\in X(F_2)$ and $f_3\in X(F_3)$.
\end{lemma}




By scaling invariance, we can assume that $|F_3|=1$.
We partition $\bS$ into two subsets $\bS^{(1)}$ and $\bS^{(2)}$, where
\begin{equation}\label{defS(1)}
\bS^{(1)}= \{(j,n)\in \bS: |\w'_{2,j}|\leq |\w'_{1,j}|/10\,\, {\rm or}\,\,
 |\w'_{1,j}|\leq |\w'_{2,j}|/10\}
\end{equation}
\begin{equation}\label{defS(2)}
\bS^{(2)}= \bS\backslash \bS^{(1)}\,.
\end{equation}
We should change the definitions of sizes of trees in $\bS$.

\begin{definition}\label{defofnorm111}
Let $(j,n)\in\bS$ and $\ell\in\{1, 2,3\}$. Define a semi-norm
$\big\|f_\ell\big\|_{j,n}$  by
\begin{equation}\label{defnormm}
\big\|f_{\ell}\|_{j,n} =\frac{1}{|I_s|^{1/2}}
\big\|\Id^{**}_{j,n}f_{\ell, j,0}\big\|_2
   + \frac{1}{|I_s|^{1/2}}\big\|2^{-k_{j\ell}}\Id^{**}_{j,n}
 Df_{\ell, j, 0}\big\|_2\,,
\end{equation}
where $Df_{\ell,j, 0}$ is the derivative of $f_{\ell,j,0}$.
\end{definition}

\begin{definition}
For $\ell\in\{1,2, 3\}$ and a tree  $\bT$,
let $(j_\bT, n_\bT)$ be the top of the tree $\bT$.  And define
\begin{equation}\label{defdelta1111}
\Delta^*_{\ell}(\bT)(x)=\bigg(\sum_{(j,n)\in\bT}
\big|\Id^{*}_{j,n}f_{\ell, j,0}(x)\big|^2 \bigg)^{1/2}\,.
\end{equation}

If $\bT $ is a tree in $\bS^{(1)}$, we
 define
\begin{equation}\label{size12p2}
\size_{\ell}(\bT) = \frac{1}{|I_\bT|^{1/2}}\big\|\Delta^*_{\ell}(\bT)
\big\|_2 + \big\|f_\ell\big\|_{j_\bT,n_\bT}\,,
\end{equation}
for all $\ell\in\{1,2,3\}$.

If $\bT $ is a tree in $\bS^{(2)}$, define $\size_\ell(\bT)$
by (\ref{size12p2}) only for $\ell\in\{1,2\}$. For $\ell=3$,
we define the size by
\begin{equation}\label{size3p2}
\size_{3}(\bT) = \big\|f_3\big\|_{j_\bT,n_\bT}\,,
\end{equation}

Let ${\bf P}$ be a subset of $\bS$. Define the $\ell$-$\size^*$
of $\bT$ by
\begin{equation}
\size^*_\ell({\bf P})=\sup_{\bT: \bT\subset {\bf P}}\size_\ell(\bT)\,,
\end{equation}
where $\bT$ ranges over all trees in ${\bf P}$.
\end{definition}

One should notice that for $\La_{\bS^{(1)},m}$ we have a uniform estimate
for $p_1, p_2>2$ and $1<r<2$. We state it as follow

\begin{proposition}\label{uniestp2good}
Let $p_1, p_2>2$ and $1<r<2$ with $1/p_1+1/p_2=1/r$.
Let $f_1\in L^{p_1}$, $f_2\in L^{p_2}$ and $f_3\in L^{r'}$. Then
\begin{equation}\label{goodp2est}
\big|\La_{\bS^{(1)},m}(f_1, f_2, f_3)\big|
\leq C\|f_1\|_{p_1}\|f_2\|_{p_2}\|f_3\|_{r'}\,,
\end{equation}
where $C$ is independent of $m$, $f_1, f_2, f_3$.
\end{proposition}
\begin{proof}
We do not need  time frequency analysis for this proposition.
The key point is that when $s\in\bS^{(1)}$ the support of Fourier
transform of $f_{3, j, 0}$ is away from the origin so that we can apply
Littlewood-Paley Theorem  for the square function generated by  $f_{3, j, 0}$'s.
Clearly $|\La_{\bS^{(1)}, m}|$ is estimated by
$$
\int_{\mathbb R} \sum_j\prod_{\ell=1}^3 f_{\ell, j, 0}(\Tr_{\ell, j, m}(x)) dx\,.
$$
By
H\"older inequality, we dominate $|\La_{\bS^{(1)}, m}|$ by
$$
\bigg\| \bigg( \sum_j \big|f_{1,j,0}\circ\Tr_{1,j,m}\big|^{p_1}\bigg)^{1/p_1}
\bigg\|_{p_1}
\bigg\|\bigg( \sum_j \big|f_{2,j,0}\circ\Tr_{2,j,m}\big|^{p_2}\bigg)^{1/p_2}
\bigg\|_{p_2}
\bigg\|\bigg( \sum_j \big|f_{3,j,0}\big|^{r'}\bigg)^{1/r'}\bigg\|_{r'}\,.
$$
By a change of variables, it is clear that for $\ell=1, 2$,
$$
\bigg\| \bigg( \sum_j \big|f_{\ell,j,0}\circ\Tr_{\ell,j,m}\big|^{p_\ell}\bigg)^{1/p_\ell}
\bigg\|_{p_\ell}
= \bigg\| \bigg( \sum_j \big|f_{\ell,j,0}\big|^{p_\ell}\bigg)^{1/p_\ell}
\bigg\|_{p_\ell}\,.
$$
Notice the elementary inequality
$$
\bigg(\sum_{j}|a_j|^q\bigg)^{1/q} \leq \bigg(\sum_{j}|a_j|^2\bigg)^{1/2}
$$
holds for $q \geq 2$.
We thus dominate $|\La_{\bS^{(1)}, m}|$ by
$$
\bigg\| \bigg( \sum_j \big|f_{1,j,0}\big|^{2}\bigg)^{1/2}
\bigg\|_{p_1}
\bigg\|\bigg( \sum_j \big|f_{2,j,0}\big|^{2}\bigg)^{1/2}
\bigg\|_{p_2}
\bigg\|\bigg( \sum_j \big|f_{3,j,0}\big|^{2}\bigg)^{1/2}\bigg\|_{r'}\,.
$$
Now Littlewood-Paley theorem yields the desired estimate
(\ref{goodp2est}). This proves the proposition.
\end{proof}

We now use  time frequency analysis to prove Lemma \ref{uniestp2}.
Although we only need to estimate $\La_{\bS^{(2)}, m}$ due to Proposition
\ref{uniestp2good}, we still write a proof for both of
$\La_{\bS^{(1)}, m}$ and $\La_{\bS^{(2)}, m}$.

We first prove the size estimate for a single tree, that is,
\begin{equation}\label{atreesize}
\big|\La_{\bT, m}(f_1, f_2, f_3)\big|\leq C \prod_{\ell=1}^3
\size^*_\ell(\bT)|I_\bT|\,.
\end{equation}

We only prove the case when $\bT$ is a tree in $\bS^{(2)}$ for (\ref{atreesize})
since the other case is similar.
In this case $2^{k_{j\ell}}\sim 2^{k_j}$ for all $\ell$ in $\{1,2,3\}$.
We thus dominate $|\La_{\bT,m}|$ by
$$
\int_{\mathbb R} \sup_{(j,n)\in\bT}\big|
(\Id^{**}_{j,n}f_{3,j,0})\circ\Tr_{\ell,j,m}(x)\big|
 \prod_{\ell\neq 3} \bigg( \sum_{(j,n)\in\bT}\big|
(\Id^{**}_{j,n}f_{\ell,j,0})\circ\Tr_{\ell,j,m}(x)\big|^2 \bigg)^{1/2} dx\,.
$$
By the definition of $\Delta_\ell$ and H\"older inequality, we estimate
$|\La_{\bT,m}|$ by
$$
\sup_{(j,n)\in\bT}\big\|F^*_{3,j, 0}\big\|_\infty
 \big\|\Delta^*_{1}(\bT)\big\|_2 \big\|\Delta^*_{2}(\bT)\big\|_2\,,
$$
where $F^*_{3,j,0}=\Id^{**}_{j,n}f_{3,j,0}$.
 Notice that Lemma \ref{goodbmo1} holds for the semi-norm. Thus  we have
$$
\big\| F^*_{3,j,0}\big\|_\infty\leq \size^*_3({\bT})\,.
$$
Clearly the definition of size yields
$$\|\Delta_\ell(\bT)\|_2\leq \size^*_\ell({\bT})|I_\bT|^{1/2} \,$$
for $\ell\in\{1,2\}$.
Putting all of them together, we obtain (\ref{atreesize}).\\

\begin{lemma}\label{treeout}
Let $\kappa\in\{1,2\}$, $\bT$ be a tree in $\bS^{(\kappa)}$
and ${\bf P}$ be a subset of $\bS^{(\kappa)}$.
Suppose that $\bP\cap\bT=\emptyset$ and  $\bT$ is a maximal tree in $\bP\cup
\bT$. Then we have
\begin{equation}\label{treeout1}
\big|\La_{\bP\cup\bT, m}(f_1, f_2, f_3)-\La_{\bP, m}(f_1, f_2, f_3)\big|
 \leq Cm\prod_{\ell=1}^3\size^*_\ell(\bT\cup \bP)|I_\bT|\,,
\end{equation}
where $C$ is independent of $f_1, f_2, f_3, \bP$, $\bT$.
\end{lemma}
\begin{proof}
Clearly the difference $|\La_{\bP\cup\bT,m}-\La_{\bP,m}|$ is dominated by
a sum of $C\La_{\bT, m}$ and
at most finite many following trilinear forms
$$
\bigg|\int
\sum_{j\in\scl(\bT)}
  \bigg(\sum_{n\in \bT_j}F_{\ell_1,j, n, m}(x)\bigg)
 \bigg( \sum_{n\in \bP_j}F_{\ell_2,j, n, m}(x)\bigg)
 \bigg(\sum_{n\in (\bP\cup\bT)_j} F_{\ell_3,j, n, m}(x)\bigg)dx\bigg|\,,
$$
where $(\ell_1, \ell_2, \ell_3)$ is a permutation of $(1,2,3)$.
By (\ref{atreesize}), it sufficient to show that this trilinear
form can be estimated by the right hand side of (\ref{treeout}).
We only handle  the most difficult case $\ell_1=1, \ell_2=2$.
Other cases are similar. We estimate the trilinear form by
\begin{equation}\label{localI}
\sum_{j\in\scl(\bT)}\sum_{I:|I|=2^{-k_j}}
\bigg\|\bigg(\sum_{n\in\bT_j}F_{1,j,n,m}\bigg)
  \bigg(\sum_{n\in\bP_j}F_{2,j,n,m}\bigg)
   \bigg(\sum_{n\in (\bP\cup\bT)_j}F_{3,j,n,m} \bigg)
  \bigg\|_{L^{1}(I)} \,.
\end{equation}
There is at least one of indices $\ell\in\{1,2\}$ satisfying
$k_{j\ell}=k_j$. Without loss of generality, assume $k_{j1}=k_j$.
We have that for any positive integer $N$,
$$
\bigg\|\sum_{n\in\bT_j}F_{1,j,n,m}\bigg\|_{L^{\infty}(I)}
\leq \frac{C_N}{\big(1+ 2^{k_j}{\rm dist}(I(m_{j1}), I_\bT)\big)^N}
 \big\|\Id^{**}_{j,n'}f_{\ell,j,0}\big\|_{\infty}\,,
$$
where $I(m_{j1})=I+m_{j1}$ is an interval generated by shifting $I$
to the right by $m_{j1}$ and
 $n'\in(\bP\cup\bT)_{j}$ which minimizes the distance between $I_{j,n}$ and $I(m_{j1})$.
Since Lemma \ref{goodbmo1} holds for the semi-norm, we get
$$
\bigg\|\sum_{n\in\bT_j}F_{1,j,n,m}\bigg\|_{L^{\infty}(I)}
\leq \frac{C_N\size^*_{1}(\bP\cup\bT)}{\big(1+ 2^{k_j}{\rm dist}(I(m_{j1}), I_\bT)\big)^N}\,.
$$
And since $\bP\cap\bT=\emptyset$ and $\bT$ is a maximal tree in $\bP\cup\bT$,
we have
$$
\bigg\|\sum_{n\in\bP_j}F_{2,j,n,m}
  \bigg\|_{L^{2}(I)} \leq
\frac{C_N}{\big(1+ 2^{k_j}{\rm dist}(I(m_{j2}), (I_\bT)^c)\big)^N}
 \big\|\Id^{**}_{j,n'}f_{\ell,j,0}\big\|_{2}\,,
$$
which is obviously bounded by
$$
 \frac{C_N\size^*_2(\bP\cup\bT)|I|^{1/2}}{\big(1+ 2^{k_j}{\rm dist}(I(m_{j2}), (I_\bT)^c)\big)^N}
\,.
$$
Similarly, we also have
$$
\bigg\|\sum_{n\in(\bP\cup\bT)_j}F_{3,j,n,m}
  \bigg\|_{L^{2}(I)} \leq  C\size^*_3(\bP\cup\bT)|I|^{1/2}\,.
$$
Thus we estimate (\ref{localI}) by
$$
 \sum_{j\in\scl(\bT)}\sum_{I:|I|=2^{-k_j}}
 \frac{C_N\size^*_{1}(\bP\cup\bT)\size^*_2(\bP\cup\bT)
  \size^*_3(\bP\cup\bT)|I|}{\big(1+ 2^{k_j}{\rm dist}(I(m_{j1}), I_\bT)\big)^N
 \big(1+ 2^{k_j}{\rm dist}(I(m_{j2}), (I_\bT)^c)\big)^N}\,.
$$
Let $j_\bT$ be the index for the top of $\bT$. If $j_\bT+10m \geq j\geq
j_\bT$, we only have at most $10m$ different values for $j$. Notice that
if $I(m_{j1})\subset (I_\bT)^c$, then we can replace ${\rm dist}(I(m_{j1}),
I_{\bT})$ by ${\rm dist}(I(m_{j1}), \partial I_{\bT})$. Thus if we only sum
$j$ from $j_\bT$ to $j_\bT+10m$ we get that (\ref{localI}) is dominated by
$$
 Cm\prod_{\ell=1}^3\size^*_\ell(\bP\cup\bT)|I_\bT| \,.
$$
The remaining thing we need to deal with is to sum all $j>j_\bT+10m$.
The main difficulty is the case $I(m_{j1})\nsubseteq (I_{\bT})^c$
and $I(m_{j2})\nsubseteq I_\bT$, because in other cases we gain
$\big(1+2^{k_j}{\rm dist}(I(m_{j\ell}), \partial I_\bT)\big)^{-100}$
in the estimate for at least one of $\ell\in\{1,2\}$, which trivializes the estimate. We also know from the definition of $m_{j\ell}$ that
$
{\rm dist}(I(m_{j1}), I(m_{j2})) \leq 2^m|I|\,.
$
To make the difficult case happen, the interval $I$ must satisfy
${\rm dist}(I(m_{j\ell}), \partial I_\bT)
\leq 10\cdot 2^m |I|$ for both $\ell=1, 2$.  Sum $|I(m_{j\ell})|$
for all such $I$'s to get a upper bound $C 2^m 2^{-k_j} $.
Then summing these upper bounds for all $j>j_\bT+10m$ we get
a bound $C 2^{-8m}|I_\bT|$. Therefore we estimate (\ref{localI})
by $Cm\prod_{\ell=1}^3\size^*_\ell(\bP\cup\bT)|I_\bT|$. This proves the lemma.
\end{proof}

Lemma \ref{prilem} still holds for the  sizes of trees defined in 
Subsection \ref{subpl}.  Let $\kappa\in\{1,2\}$.
Applying this organization lemma inductively for $\bS^{(\kappa)}$,
we decompose
\begin{equation}\label{gooddep2}
\bS^{(\kappa)}=\bigcup_{\sigma} \bS^{(\kappa)}_\si\,,
\end{equation}
where  $\sigma $ ranges over all possible dyadic numbers,
$\bS^{(\kappa)}_\si=\cup_{\bT\in\mathcal F_\si^{(\kappa)}}\bT$
such that $\mathcal F_\si^{\kappa} $ is a collection of  maximal trees
with
\begin{equation}\label{ctestp2}
 \ct(\bS_\si^{(\kappa)})\leq C\si^{-2}\,,
\end{equation}
and
\begin{equation}\label{sizekaestp2}
\size^*_\ell(\bS_\si^{(\kappa)})\leq \si|F_\ell|^{1/2}
\end{equation}
holds for all $\ell\in\{1,2,3\}$.

Notice that Lemma \ref{Lpest} holds for the new sizes of trees defined in Subsection 
\ref{subpl}.
 We thus can also sharpen the upper bound in
the size estimate for $\bS_\si^{(\kappa)}$ by
\begin{equation}\label{sizekaest123p2}
\size^*_\ell(\bS_\si^{(\kappa)})\leq \min\{1, \si|F_\ell|^{1/2}\}\,.
\end{equation}
Hence by Lemma \ref{treeout} we estimate $\La_{\bS, m}$ by
$$
\sum_{\kappa=1}^2
 \sum_{\si}\sum_{\bT\in\mathcal F_\si^{(\kappa)}}
 m\prod_{\ell=1}^3\size^*_{\ell}(\bS_\si^{(\kappa)})|I_\bT|
\,.
$$

Applying (\ref{sizekaest123p2}) and (\ref{ctestp2}), we thus obtain
\begin{equation}\label{finalestpp2}
\big|\La_{\bS,m}(f_1, f_2, f_3)\big|\leq
 Cm\sum_{\sigma }\si^{-2} \min\{1, \si|F_1|^{1/2}\}
   \min\{1, \si|F_2|^{1/2}\}\min\{1, \si\}\,,
\end{equation}
which clearly implies (\ref{largepest2}).
Therefore we complete the proof of Lemma {\ref{uniestp2}}.\\

\subsection{A truncated trilinear form}

First by a change of variable, we write $\La_{L_1, L_2, M_1, M_2, m}$
as
\begin{equation}\label{chanlam}
 \La_{L_1, L_2, M_1 M_2, m}(f_1, f_2, f_3) =
 \int \sum_{j}\prod_{\ell=1}^{3}f_{\ell,j,0}\big(\ti\Tr_{\ell,j,m}(x)\big)dx\,,
\end{equation}
where $\ti\Tr_{1,j,m}(x)=\Tr_{1,j,m}-m_{j2}$, $\ti\Tr_{2,j,m}(x)=x$,
$\ti\Tr_{3,j,m}(x)=x-m_{j2}$. \\

To prove Lemma \ref{uniestp110}, we have to set up our
time-frequency decomposition in a slightly different way for 
technical reasons. Recall that $\psi$ is a nonnegative Schwartz
function such that $\wh\psi$ is supported in $[-1/100, 1/100]$ and
$\wh\psi(0)=1$. And $\psi_k(x)=2^k\psi(2^kx)$. Let $\Omega$ be the
set defined as in (\ref{defOmega}). As before, $k_{j\ell}$ is an
integer such that $2^{k_{j\ell}}\sim |\w'_{\ell,j}|$ for for
$\ell\in\{1,2,3\}$ and $k_j=\min\{k_{j\ell}\}$.  For a very small
positive number $\e$, we define
\begin{equation}\label{defomegajl}
 \Omega_{j} = \{x\in\Omega: {\rm dist}(
x, \Omega^c)\geq  2^{\e^2m}2^{-k_j}\}\,.
\end{equation}
\begin{equation}\label{defpsijl}
\psi_{j1}=\psi_{j2}=\psi_{j3}= \Id_{(\Omega_{j})^c}*\psi_{k_{j}}(x)\,.
\end{equation}
$\Omega_j$, $\psi_{j\ell}$ depend on $m, \e$ but this dependence
is suppressed for notational convenience.
A truncated trilinear form is defined by
\begin{equation}\label{deftruntri}
\La_{\Omega, m}(f_1, f_2, f_3)= \int\sum_{j\in\mathbb Z}
\prod_{\ell=1}^3\psi_{j\ell}(x)
 f_{\ell, j, 0}\big(\ti\Tr_{\ell, j, m}(x)\big)dx\,.
\end{equation}
Heuristically, $\psi_{j\ell}$ can be considered as $\Id_{(\Omega_{j})^c}$
since it is a smooth approximation of $\Id_{(\Omega_j)^c}$. In time
space, $\Omega_j$ is an exceptional set which can be removed. 
we can handle it  well.  The technical details about this can be found in
Section \ref{para2}. In order to get $2^{\e m}$ instead of $2^m$ in the estimates, 
we have to remove only a smaller set. Here is the lemma which allows us to do so.

\begin{lemma}\label{trunlem}
Let $F_1$, $F_2$, $F_3$ be measurable sets. Let $F'_3=F_3\backslash\Omega$.
Then
\begin{equation}\label{trundiff}
 \big| \big(\La_{L_1, L_2, M_1, M_2, m}-\La_{\Omega, m}\big)(f_1, f_2, f_3)\big| \leq C2^{-100m}\min\big\{1, |F_1|^{1/p}\big\}\min\big\{1,
|F_2|^{1/p}\big\}\,
\end{equation}
holds for all functions $f_1\in X(F_1), f_2\in X(F_2), f_3\in X(F'_3)$, 
where $C$ is a constant independent of $L_1, L_2$, $M_1, M_2, m$, $f_1, f_2, f_3$, $F_1, F_2, F_3$.
\end{lemma}
\begin{proof}
The difference $|\La_{L_1, L_2, M_1, M_2, m}-\La_{\Omega,m}|$ is
dominated by
$$
\int\sum_j\big|1-\prod_{\ell=1}^3\psi_{j\ell}(x)|\bigg|\prod_{\ell=1}^3f_{\ell,j,0} \big(\ti\Tr_{\ell,j,m}(x)\big)\bigg| dx\,.
$$
Clearly,
$$
\big|1-\prod_{\ell=1}^3\psi_{j\ell}(x)\big|
\leq 3\sum_{\ell=1}^3\big|1- \psi_{j\ell}(x)\big|
$$
For $\ell=\{1,2\}$, by the definition of $\Omega$, we have for any positive integer $N$,
\begin{eqnarray*}
  &   &  \big|f_{\ell, j,0}\big(\ti\Tr_{\ell,j,m}(x)\big) \big| \\
 & \leq & \int \frac{C_N|f_\ell(y)| 2^{k_{j\ell}}}{ \big(
1+2^{k_{j\ell}}|\ti \Tr_{\ell,j,m}(x)-y| \big)^N} dy\\
 & \leq & C 2^{2m}(1+ 2^{k_{j\ell}}{\rm dist}(\ti\Tr_{\ell,j,m}(x), \Omega^c))^2
   \min\{1, |F_\ell|^{1/p}\}\,.
\end{eqnarray*}
Since $f_3\in X(F'_3)$, we obtain that
\begin{equation}\label{estf3trun}
 \big|f_{3,j,0}(x)\big|\leq \frac{C_N}{\big(1+ 2^{k_{j3}} {\rm dist}(
\ti\Tr_{3,j,m}(x),   \Omega^c)\big)^N}\,.
\end{equation}
Thus by the fact that $2^{k_{j3}}\sim \max \{2^{k_{j\ell}}\}$,
 $k_{j2}> k_{j1}+m$ and the definition of $\Omega_j$,
 the difference in the left hand side of (\ref{trundiff}) is estimated by
\begin{eqnarray*}
 & &
 \sum_j\int\int_{\Omega_{j}}\frac{2^{k_{j}}}{\big(
 1+2^{k_{j}}|x - y|\big)^N} dy
 \frac{C_N2^{4m}\min\{1, |F_1|^{1/p}\}\min\{1, |F_2|^{1/p}\} } {\big(1+ 2^{k_{j3}} {\rm dist}(\ti\Tr_{3,j,m}(x), \Omega^c)\big)^N} dx\\
 & \leq & \sum_j\int_{\Omega_{j}}
 \frac{C_N2^{4m}\min\{1, |F_1|^{1/p}\}\min\{1, |F_2|^{1/p}\}}
 {\big(1+2^{k_{j}}{\rm dist}(y, \Omega^c)\big)^N} dy
\\
 & \leq & C2^{-100m}\min\{1, |F_1|^{1/p}\}\min\{1, |F_2|^{1/p}\}\,.
\end{eqnarray*}
Therefore we finish the proof.
\end{proof}

By this lemma, we only need to consider $\La_{\Omega,m}$.
For $\bS\subset \mathbb Z(\gamma)\times \mathbb Z$ we define
\begin{equation}\label{deftrunlam}
\La_{\bS, \Omega, m}(f_1, f_2, f_3)=\int_{\mathbb R}\sum_{j\in\mathbb Z}
\prod_{\ell=1}^{3}\sum_{n\in\bS_j}\ti F_{\ell,j, n, m}(x) dx\,,
\end{equation}
where $\ti F_{\ell,j,n,m}$ is defined by
\begin{equation}\label{tiFljnm}
\ti F_{\ell,j,n,m}(x)= \psi_{j\ell}(x)\Id^{*}_{j,n}
\big(\ti\Tr_{\ell,j,m}(x)\big)f_{\ell,j,0}\big(\ti\Tr_{\ell,j,m}(x)\big) \,.
\end{equation}
As before we only need to consider the trilinear form (\ref{deftrunlam}).
To prove Lemma \ref{uniestp110}, it is sufficient to show the
following lemma due to Lemma \ref{trunlem}.

\begin{lemma}\label{uniestp11}
Let $\e$ be any positive number, $1<p<2$ and $F_1, F_2, F_3$ be
measurable sets in $\mathbb R$ such that $|F_3|=1$. There is a
subset $F_3'\subset F_3$ with $|F'_3|\geq |F_3|/2$ such that for all
$p_1, p_2\geq p$ with $1/p_1+1/p_2\geq 1$, and all functions $f_1\in
X(F_1)$, $f_2\in X(F_2)$,  $f_3\in X(F_3)$, the following inequality
holds.
\begin{equation}\label{psmallest11}
\big|\La_{\bS, \Omega, m}(f_1, f_2, f_3)\big|\leq C2^{\e
m}|F_1|^{1/p_1}|F_2|^{1/p_2}
 \, ,
\end{equation}
where $C$ is a constant independent of $\bS$, $F_1, F_2, F_3, f_1,
f_2, f_3, L_1, L_2, M_1, M_2, m$.
\end{lemma}

\subsection{Preliminary Lemmata}

To prove Lemma \ref{uniestp11},
we should change the definitions of size of a tree in $\bS$
and set up some lemmata first.

\begin{definition}\label{defofnormp11}
Let $(j,n)\in\bS$ and $\ell\in\{1, 2,3\}$.
Let $\psi_{j\ell}^*$ be the function
\begin{equation}\label{defofpsistar}
\psi^*_{j\ell}(x)= \int_{(\Omega_j)^c}\frac{2^{k_j}}{\big(
 1 +2^{2k_j}|x-y|^2  \big)^{200}}dy
\end{equation}
Define a semi-norm
$\big\|f_\ell\big\|_{j,n,m}$  by
\begin{equation}\label{defnormmp11}
\big\|f_{\ell}\|_{j,n,m} =\frac{1}{|I_s|^{1/p}}
\big\|\Id^{**}_{j,n}\big(\psi^*_{j\ell}\circ \ti\Tr^{-1}_{\ell,j,m}\big)
f_{\ell, j,0}\big\|_p
   + \frac{1}{|I_s|^{1/2}}\big\|2^{-k_{j\ell}}\Id^{**}_{j,n}
 \big(\psi^*_{j\ell}\circ \ti\Tr^{-1}_{\ell,j,m}\big)Df_{\ell, j, 0}\big\|_p\,,
\end{equation}
where $\ti\Tr^{-1}_{\ell,j,m}$ is the inverse of $\ti\Tr_{\ell,j,m}$
and  $Df_{\ell,j, 0}$ is the derivative of $f_{\ell,j,0}$.
\end{definition}

\begin{definition}
For $\ell\in\{1,2\}$ and a tree $\bT$,
let $(j_\bT, n_\bT)$ be the top of the tree $\bT$.  And let
$\Delta^*_{\ell,m}(\bT)$ be defined by

\begin{equation}\label{defdeltap11m}
\Delta^*_{\ell,m}(\bT)(x)=\bigg(\sum_{(j,n)\in\bT}
\big|\Id^{*}_{j,n}(x)
\big(\psi^*_{j\ell}\circ \ti\Tr^{-1}_{\ell,j,m}\big)(x)
 f_{\ell, j,0}(x)\big|^2 \bigg)^{1/2}\,.
\end{equation}

If $\bT $ is a tree in $\bS$, we
 define
\begin{equation}\label{size12p11}
\size_{\ell,m}(\bT) = \frac{1}{|I_\bT|^{1/2}}\big\|\Delta^*_{\ell,m}(\bT)
\big\|_p + \big\|f_\ell\big\|_{j_\bT,n_\bT,m}\,,
\end{equation}
for all $\ell\in\{1,2\}$.


Let ${\bf P}$ be a subset of $\bS$. Define the $(\ell,m)$-$\size^*$
of $\bT$ by
\begin{equation}
\size^*_{\ell,m}({\bf P})=\sup_{\bT: \bT\subset {\bf P}}\size_{\ell,m}(\bT)\,,
\end{equation}
where $\bT$ ranges over all trees in ${\bf P}$.

In the definition of $\psi^*_{j\ell}$, we can replace the exponent $200$
by a larger number $2^{100}$ to define a new function. Denote this function
by $\ti\psi^*_{j\ell}$. If $\Id^{*}_{j,n}$ and $\psi^*_{j\ell}$ are
replaced by $\ti\Id^*_{j,n}$ and $\ti\psi^*_{j\ell}$ respectively
in the definition $\Delta_{\ell,m}^*(\bT)$, we denote the corresponding
function by $\Delta_{\ell,m}(\bT)$.
\end{definition}

\begin{lemma}\label{Lpestm}
Let  $1<q<\infty$, $\ell\in\{1,2,3\}$ and $\bT$ be a tree in $\bS$.
Then
\begin{equation}\label{Lpest1m}
\big\|\Delta^*_{\ell,m}(\bT)\big\|_{q}\leq C \inf_{x\in
I_\bT}M_q(Mf_\ell)(x)|I_\bT|^{1/q}\,,
\end{equation}
\begin{equation}\label{Lpest2m}
\size_{\ell,m}(\bT)\leq C \min\{2^{\beta_\ell m}|F_\ell|^{1/p},
 \inf_{x\in I_\bT}M_p(Mf_\ell)(x)\}\,,
\end{equation}
where $\beta_\ell=1$ if $\ell=1$, $\beta_\ell=\e^2$ if $\ell=2$, and 
$C$ is a constant independent of $f_\ell, \bT$, $\bS$, $L_1$, $L_2$,
$M_1, M_2$.
\end{lemma}
\begin{proof}
Repeating a similar argument in the proof of (\ref{Lpest1}) and (\ref{Lpest2}),
we obtain easily (\ref{Lpest1m}) and part of (\ref{Lpest2m}). The only thing
we need to prove is
\begin{equation}\label{Lpest3m}
 \size_{\ell,m}(\bT)\leq C 2^{\beta_\ell m}|F_\ell|^{1/p}\,.
\end{equation}
Assume $2^{\beta_\ell m+10}I_\bT\subset \Omega$, otherwise (\ref{Lpest3m})
follows by the upper bound $\inf_{x\in I_\bT}M_p(Mf_\ell)(x)$. Let
$\bT_{L}$ be a  collection of all $s=(j,n)\in\bT$ such that
$ 2^{L}I_s\subset \Omega$ but $2^{L+1}I_s\nsubseteq \Omega$.
Then
$$ \bT= \bigcup_{L= [\beta_\ell m+10]}^\infty \bT_{L}  $$
Let ${\mathbb J}_L$ be the set of all time intervals $I_s$'s for $s\in\bT_L$.
Clearly, ${\mathbb J}_L$ is a set of disjoint intervals and
$\sum_{J\in{\mathbb J}_L}|J|\leq \min\{|I_\bT|, 1\}$. Thus it is sufficient
to show that for any $J\in {\mathbb J}_L$ and any $(j,n)\in\bT$ such that
$I_s=J$,
\begin{equation}\label{localLpest3m}
 \big\|\Id^{**}_{j,n}\big(\psi^*_{j\ell}\circ\ti\Tr^{-1}_{\ell,j,m}\big)\ti
f_{\ell,j,0}\big\|_p^p  \leq C_N\big( \inf_{x\in J}M_p(Mf_\ell)(x)\big)^{p} L^{-N} |J|\,
\end{equation}
holds for a large integer $N$, where
$\ti f_{\ell,j,0}$ is $f_{\ell,j,0}$ or $2^{-k_{j\ell}}Df_{\ell,j,0}$,
since the desired estimate follows by summing up all $L$'s and $J$'s.
By the definition of $\psi^*_{j\ell}$, we have
$$
\big|\Id^{**}_{j,n}(x) \psi^*_{j\ell}\circ \ti\Tr^{-1}_{\ell,j,m}(x)\big|\leq \frac{C}{ \big(1+2^{k_j}{\rm dist}(x, J)\big)^{200}
\big( 1+ 2^{k_j}{\rm dist}\big( \ti\Tr^{-1}_{\ell,j,m}(x), (\Omega_j)^c\big)  \big)^{200}}\,,
$$
which is clearly dominated by
$$
\frac{C}{ \big(1+2^{k_j}{\rm dist}(x, J)\big)^{100}
\big( 1+ 2^{k_j}{\rm dist}\big( J_{j,m}, (\Omega_j)^c\big)  \big)^{100}}\,,
$$
where $J_{j,m}$ is the interval $\{\ti\Tr_{\ell,j,m}(x): x\in J\}$. Since
$L\geq \beta_\ell m +9$, by the definition of $\ti\Tr_{\ell,j,m}$ we thus dominate $\big|\Id^{**}_{j,n}( \psi^*_{j\ell}\circ \ti\Tr^{-1}_{\ell,j,m})\big|$ by
$$
\frac{C}{ \big(1+2^{k_j}{\rm dist}(x, J)\big)^{100}
\big( 1+ 2^{k_j}{\rm dist}\big( J, (\Omega)^c\big)  \big)^{100}}\,.
$$
Thus we have
$$
\big\|\Id^{**}_{j,n}\big(\psi^*_{j\ell}\circ\ti\Tr^{-1}_{\ell,j,m}\big)\ti
f_{\ell,j,0}\big\|_p^p  \leq C
\big( \inf_{x\in J}M_p(Mf_\ell)(x)\big)^{p} L^{-100p} |J|\,,
$$
which yields (\ref{localLpest3m}). Therefore we finish the proof.
\end{proof}

\begin{lemma}\label{goodbmo1p11}
Suppose that  $s=(j,n)\in\bS$.
If $2^{k_{j\ell}}\sim 2^{k_j}$, then
\begin{equation}\label{smallboundp11}
\big\|\Id^{**}_{j,n} \big(\psi^*_{j\ell}\circ\ti\Tr^{-1}_{\ell,j,m}\big)f_{\ell,j,0}\big\|_\infty \leq
C\big\|f_\ell\big\|_{j,n,m}
\end{equation}
holds for $\ell\in\{1,2,3\}$, where $C$ is a constant independent of
$s,  f_\ell, m$, $L_1, L_2, M_1, M_2$.
\end{lemma}
\begin{proof}
 Let $\mu=\big\|f_\ell\big\|_{j,n,m}$.
 By the definition of the semi-norm, we have
\begin{equation}\label{smalldelta2p11}
\big\|\Id^{**}_{j,n}\big(\psi^*_{j\ell}\circ\ti\Tr^{-1}_{\ell,j,m}\big) f_{\ell,j, 0}\big\|_p +
\big\|{|I_{s}|}\Id^{**}_{j,n}\big(\psi^*_{j\ell}\circ\ti\Tr^{-1}_{\ell,j,m}\big)
 Df_{\ell,j, 0} \big\|_p \leq \mu |I_s|^{1/p}\,.
\end{equation}
 First we prove the BMO estimate for the function, that is
\begin{equation}\label{smallbmop11}
\big\|\Id^{**}_{j,n}\big(\psi^*_{j\ell}\circ\ti\Tr^{-1}_{\ell,j,m}\big) f_{\ell,j,0}\big\|_{BMO} \leq C\mu\,.
\end{equation}
If $|I_s|\leq |J|$, by (\ref{smalldelta2p11}) we
have
\begin{eqnarray*}
  &   & \inf_c\int_J\big| \Id^{**}_{j,n}(x)\big(\psi^*_{j\ell}\circ\ti\Tr^{-1}_{\ell,j,m}\big)(x)f_{\ell,j,0}(x)-c\big| dx \\
 & \leq & \big\|\Id^{**}_{j,n}
  \big(\psi^*_{j\ell}\circ\ti\Tr^{-1}_{\ell,j,m}\big)
  f_{\ell,j,0} \big\|_p|J|^{1-\frac{1}{p}}
 \,\leq \, \mu |I_s|^{\frac{1}{p}}|J|^{1-\frac{1}{p}}
  \, \leq \, \mu |J|\,.
\end{eqnarray*}
If $|I_s|\geq |J|$, by (\ref{smalldelta2p11}) we obtain that
\begin{eqnarray*}
 &  & \inf_c\int_J\big| \Id^{**}_{j,n}(x)
 \big(\psi^*_{j\ell}\circ\ti\Tr^{-1}_{\ell,j,m}\big)(x) f_{\ell,j,n_\ell}(x)-c\big| dx\\
 & \leq & |J|\int_J\bigg| \bigg(\Id^{**}_{j,n}
  \big(\psi^*_{j\ell}\circ\ti\Tr^{-1}_{\ell,j,m}\big)
f_{\ell,j,n_\ell}\bigg)' (x)\bigg| dx\\
 & \leq & C|J||I_s|^{-1}
 \int_J \big|\Id^{**}_{j,n}(x) \big(\psi^*_{j\ell}\circ\ti\Tr^{-1}_{\ell,j,m}\big)(x)f_{\ell,j,n_\ell}(x)  \big|dx \\
  & &  + |J|\int_J\big| \Id^{**}_{j,n}(x)
 \big(\psi^*_{j\ell}\circ\ti\Tr^{-1}_{\ell,j,m}\big)(x)
Df_{\ell,j,n_\ell}(x) \big|dx\\
 & \leq & C|J||I_s|^{-1}\big\|\Id^{**}_{j,n}\big(\psi^*_{j\ell}\circ\ti\Tr^{-1}_{\ell,j,m}\big)f_{\ell,j,n_\ell}\big\|_p|J|^{1-\frac{1}{p}} + |J|\big\| \Id^{**}_{j,n}\big(\psi^*_{j\ell}\circ\ti\Tr^{-1}_{\ell,j,m}\big)Df_{\ell,j,n_\ell}\big\|_p
 |J|^{1-\frac{1}{p}}\\
 &\leq & C\mu |J|^{2-\frac{1}{p}} |I_s|^{\frac{1}{p}-1}\,\leq\, C\mu|J|\,.
\end{eqnarray*}
Thus we get the BMO estimate (\ref{smallbmop11}). Interpolating (\ref{smallbmop11})
and (\ref{smalldelta2p11}), we have for any $p\leq q <\infty$,
$$
\big\|\Id^{**}_{j,n}\big(\psi^*_{j\ell}\circ\ti\Tr^{-1}_{\ell,j,m}\big)f_{\ell,j,n_\ell}\big\|_q\leq C\mu |I_s|^{1/q}\,.
$$
Notice that an integration by part and H\"older inequality yield that
$$
\big\|\Id^{**}_{j,n}\big(\psi^*_{j\ell}\circ\ti\Tr^{-1}_{\ell,j,m}\big) f_{\ell,j,n_\ell}\big\|_\infty \leq
 \big\| \Id^{**}_{j,n}\big(\psi^*_{j\ell}\circ\ti\Tr^{-1}_{\ell,j,m}\big) f_{\ell,j,n_\ell}  \big\|_{p'}^{1/2}
 \big\| \big( \Id^{**}_{j,n}\big(\psi^*_{j\ell}\circ\ti\Tr^{-1}_{\ell,j,m}\big) f_{\ell,j,n_\ell}  \big)'  \big\|_p^{1/2},
$$
where $1/p+1/p'=1$. Hence the desired estimate (\ref{smallboundp11})
 follows by (\ref{smalldelta2p11}) and the $L^{p'}$ estimate for the functions.
\end{proof}

\begin{lemma}\label{goodbmo2p11}
For any tree $\bT $ in $\bS$,
let
\begin{equation}\label{tideltalm}
\ti\Delta_{\ell,m}(\bT)(x)=\bigg(\sum_{(j,n)\in \bT}
\big| \ti\Id^{*}_{j,n}\big(\ti\Tr_{\ell,j,m}(x)\big)
 \ti\psi^{*}_{j\ell}(x)f_{\ell,j,0}\big(\ti\Tr_{\ell,j,m}(x)\big|^2\bigg)^{1/2}\,.
\end{equation}
Then for $\ell=1$ we have
\begin{equation}\label{smallBMOp110}
\big\| \Delta_{\ell,m}(\bT) \big\|_{BMO}\leq C \size^*_\ell(\bT) \,,
\end{equation}
\begin{equation}\label{smallBMOp11}
 \big\| \ti\Delta_{\ell,m}(\bT) \big\|_{BMO}  \leq C m\,\size^*_\ell(\bT) \,,
\end{equation}
\begin{equation}\label{Lqestp11}
 \big\| \ti\Delta_{\ell,m}(\bT) \big\|_q\leq
 C m^{1-2/q}\size^*_\ell(\bT)|I_\bT|^{1/q} \,,
\end{equation}
where $q\geq 2$ and $C$ is a constant independent of $ \bT, \bS, L_1, L_2, M_1, M_2,
 f_\ell, n_\ell$.
\end{lemma}
\begin{proof}
(\ref{smallBMOp110}) can be obtained by a routine way as  we did
for Lemma \ref{goodbmo2}. We omit the details.  We should only prove (\ref{smallBMOp11}).  (\ref{Lqestp11}) is a simple consequence of (\ref{smallBMOp110}), (\ref{smallBMOp11}) and an interpolation argument.

Clearly by a change of variables $\big\|\Delta_{\ell,m}(\bT) \big\|_2
 = \big\|\ti\Delta_{\ell,m}(\bT) \big\|_2$.
Thus  (\ref{smallBMOp110}) and an interpolation yield
\begin{equation}\label{L2estmp11}
\big\|\ti\Delta_{\ell,m}(\bT) \big\|_2\leq
 C \size^*_\ell(\bT)|I_\bT|^{1/2} \,.
\end{equation}

Let $\mu=\size^*_\ell(\bT)$.
Let $J$ be a dyadic interval and $\bT_J=\{s\in\bT: I_s\subseteq 3J\}$.
We then dominate $\inf_c\int_J\big|\Delta_\ell(\bT)(x) - c\big|dx$ by a sum of
the following three parts.
$$
  \int_J \bigg(\sum_{s\in \bT_J} \big|\ti\Id^{*}_{j,n}\big(\ti\Tr_{\ell,j,m}(x)\big) \ti\psi^{*}_{j\ell}(x)f_{\ell,j,0}\big(\ti\Tr_{\ell,j,m}(x)\big)
\big|^2\bigg)^{1/2} dx\,,
$$
$$
 \int_J \bigg(\sum_{\substack{s\in \bT\backslash\bT_J\\|I_s|\leq |J|}} \big|
 \ti\Id^{*}_{j,n}\big(\ti\Tr_{\ell,j,m}(x)\big)
 \ti\psi^{*}_{j\ell}(x)f_{\ell,j,0}\big(\ti\Tr_{\ell,j,m}(x)\big)
 \big|^2\bigg)^{1/2} dx\,,
$$
and
$$
 \inf_c\int_J \bigg|\bigg(\sum_{\substack{s\in \bT\backslash\bT_J\\|I_s|>|J|}} \big|\ti\Id^{*}_{j,n}\big(\ti\Tr_{\ell,j,m}(x)\big)
 \ti\psi^{*}_{j\ell}(x)f_{\ell,j,0}\big(\ti\Tr_{\ell,j,m}(x)\big)
\big|^2\bigg)^{1/2}-c\bigg| dx\,.
$$

$\bT_J$ can be decomposed to a union of trees $\bT_{J,k}$'s
such that the time intervals $I_{\bT_{J,k}}$'s are disjoint
and all of them are contained in $3J$. Using Cauchy-Schwarz inequality, the first part is estimated by
$$
\big( \sum_{k} \big\|\ti\Delta_{\ell,m}(\bT_{J,k}) \big\|_2^2 \big)^{1/2}
|J|^{1/2}\,.
$$
Appying (\ref{L2estmp11}), we dominated the first part by $C\mu |J|$.

Since $p\leq 2$ we estimate the second part by
\begin{eqnarray*}
  & &  \bigg\| \bigg(\sum_{\substack{s\in \bT\backslash\bT_J\\|I_s|\leq |J|}} \big|(\ti\Id^{*}_{j,n}\circ\ti\Tr_{\ell,j,m})
\ti\psi^*_{j\ell}
 (f_{\ell,j,n_\ell}\circ\ti\Tr_{\ell,j,m} )\big|^2\bigg)^{1/2}   \bigg\|_{L^p(J)}|J|^{1-\frac{1}{p}}\\
 & \leq &  \bigg(\sum_{\substack{s\in \bT\backslash\bT_J\\|I_s|\leq |J|}} \big\| (\ti\Id^{*}_{j,n}\circ\ti\Tr_{\ell,j,m})
\ti\psi^*_{j\ell} (f_{\ell,j,n_\ell}\circ\ti\Tr_{\ell,j,m})
\big\|^p_{L^p(J)}\bigg)^{1/p}
|J|^{1-\frac{1}{p}}\\
  & \leq   & \bigg(\sum_{\substack{s\in \bT\backslash\bT_J\\|I_s|\leq |J|}}
 \frac{ C
 \big\| (\Id^{**}_{j,n}\circ\ti\Tr_{\ell,j,m})
 \psi^*_{j\ell} (f_{\ell,j,n_\ell}\circ\ti\Tr_{\ell,j,m} )    \big\|^p_p  }
 {\big( 1+ |I_s|^{-1}{\rm dist}(J, \ti\Tr^{-1}_{\ell,j,m}(I_s))\big)^{100}}\bigg)^{1/p}|J|^{1-\frac{1}{p}}\\
 &\leq & \mu \bigg(\sum_{\substack{s\in \bT\backslash\bT_J\\|I_s|\leq |J|}}
 \frac{ C |I_s| }
 {\big( 1+ |I_s|^{-1}{\rm dist}(J, \ti\Tr^{-1}_{\ell,j,m}(I_s))\big)^{100}}   \bigg)^{1/p} |J|^{1-\frac{1}{p}}\,,
\end{eqnarray*}
where $\ti\Tr^{-1}_{\ell,j,m}(I_s)$ is the interval
$\{\ti\Tr^{-1}_{\ell,j,m}(x): x\in I_s\}$. Observe that
if $|I_s|\leq 2^{-m-10}|J|$ and $s\in\bT\backslash\bT_J$, then
${\rm dist}(J, \ti\Tr^{-1}_{\ell,j,m}(I_s))\sim {\rm dist}(J, I_s)$.
Thus summing for all $s$ in this case, we get the desired estimate $C\mu|J|$.
In the remaining case, there are only $10m$ different scales for
$|I_s|$'s since $s$'s satisfy $2^{-m-10}|J|< |I_s|\leq |J|$.
The worst situation is that when $\ti\Tr^{-1}_{\ell,j,m}(I_s)
\cap J\neq\emptyset$, because otherwise ${\rm dist}(J, \ti\Tr^{-1}_{\ell,j,m}(I_s))$ can be replaced by
${\rm dist}(\partial J, \ti\Tr^{-1}_{\ell,j,m}(I_s))$ and thus
the desired estimate follows. But in this situation,
$\ti\Tr^{-1}_{\ell,j,m}(I_s)$ must be a subset of $3J$ since
$|I_s|\leq |J|$. For all $\ti\Tr^{-1}_{\ell,j,m}(I_s) \subset 3J$ with a fixed scale, the sum of $|I_s|$'s is no more than $3|J|$. Summing for at most
$10m$ different scales, we thus get the upper bound
$Cm\mu|J|$. Hence the second part is dominated by $Cm\mu|J|$.

The third part is estimated by
\begin{eqnarray*}
  &  &  \bigg(\inf_c\int_J \bigg|\bigg(\sum_{\substack{s\in \bT\backslash\bT_J\\|I_s|>|J|}}
\big|\ti\Id^{*}_{j,n}\big(\ti\Tr_{\ell,j,m}(x)\big)
 \ti\psi^{*}_{j\ell}(x)f_{\ell,j,0}\big(\ti\Tr_{\ell,j,m}(x)\big)
 \big|^2\bigg)^{1/2}-c\bigg|^2dx\bigg)^{1/2}|J|^{1/2}\\
  & \leq &\bigg(\inf_c\int_J \bigg|\sum_{\substack{s\in \bT\backslash\bT_J\\|I_s|>|J|}} \big|\ti\Id^{*}_{j,n}\big(\ti\Tr_{\ell,j,m}(x)\big)
 \ti\psi^{*}_{j\ell}(x)f_{\ell,j,0}\big(\ti\Tr_{\ell,j,m}(x) \big)\big|^2-c\bigg|dx\bigg)^{1/2}|J|^{1/2}\\
  & \leq & C\bigg(\int_J \sum_{\substack{s\in \bT\backslash\bT_J\\|I_s|>|J|}} \bigg|
  \bigg(\big|\ti\Id^{*}_{j,n}\big(\ti\Tr_{\ell,j,m}(x)\big)
 \ti\psi^{*}_{j\ell}(x)f_{\ell,j,0}\big(\ti\Tr_{\ell,j,m}(x) \big)\big|^2\bigg)'\bigg|dx  \bigg)^{1/2}|J|\,,
\end{eqnarray*}
which is dominated by a sum of following two terms,
$$
R_1= C\bigg(\int_J \sum_{\substack{s\in
\bT\backslash\bT_J\\|I_s|>|J|}}
  |I_s|^{-1}\big|\ti\Id^{*}_{j,n}\big(\ti\Tr_{\ell,j,m}(x)\big)
 \ti\psi^{*}_{j\ell}(x)f_{\ell,j,0}\big(\ti\Tr_{\ell,j,m}(x) \big)
 \big|^2 dx  \bigg)^{1/2}|J|\,,
$$
and
$$
R_2= C\bigg(\int_J \sum_{\substack{s\in
\bT\backslash\bT_J\\|I_s|>|J|}}\big|\ti\Id^{*}_{j,n}\big(\ti\Tr_{\ell,j,m}(x)\big)
 \ti\psi^{*}_{j\ell}(x)f_{\ell,j,0}\big(\ti\Tr_{\ell,j,m}(x) \big)\big|
\big| G_{\ell,j,m}(x)
 \big| dx \bigg)^{1/2}|J|\,,
$$
where $G_{\ell,j,m}$ is the function defined by
$$
 G_{\ell,j,m}(x)= \ti\Id^{*}_{j,n}\big(\ti\Tr_{\ell,j,m}(x)\big)
 \ti\psi^{*}_{j\ell}(x)  Df_{\ell,j,0}\big(\ti\Tr_{\ell,j,m}(x)\big)
$$

By Lemma \ref{goodbmo1p11}, we see that for any
$q\geq p$,
$$
 \big\|(\Id^{**}_{j,n}\circ\ti\Tr_{\ell,j,m} )\psi^{*}_{j\ell}
  (f_{\ell,j,n_\ell}\circ\ti\Tr_{\ell,j,m} )\big\|_q\leq
 C\mu|I_s|^{1/q}\,.
$$
Thus, by H\"older inequality, the first term $R_1$ is estimate by
\begin{eqnarray*}
 &  & C\bigg( \sum_{\substack{s\in
\bT\backslash\bT_J\\|I_s|>|J|}}\frac{|I_s|^{-1}
\big\|(\Id^{**}_{j,n}\circ\ti\Tr_{\ell,j,m} )\psi^{*}_{j\ell}
  (f_{\ell,j,n_\ell}\circ\ti\Tr_{\ell,j,m} )    \big\|_4^2 |J|^{1/2}
}{\big(1+|I_s|^{-1}{\rm
dist}(J, \ti\Tr^{-1}_{\ell,j,m}(I_s))\big)^{100}}
\bigg)^{1/2}|J| \\
 & \leq & C\mu\bigg( \sum_{\substack{s\in
\bT\backslash\bT_J\\|I_s|>|J|}}\frac{|I_s|^{-1/2}|J|^{1/2}}
{\big(1+|I_s|^{-1}{\rm dist}(J, \ti\Tr^{-1}_{\ell,j,m}(I_s)  )\big)^{100}}\bigg)^{1/2}|J| \,\,\leq \, C\mu |J|\,.
\end{eqnarray*}
It is obvious  by the fact $2^{k_{j\ell}}\sim 2^{k_j}$
 when $\ell=1$  and the definition of the semi-norm that
\begin{equation}
 \|G_{\ell,j,m}\|_p\leq \|f_\ell\|_{j,n,m}|I_s|^{1/p-1}\,.
\end{equation}
Thus the second term $R_2$ is estimated by
\begin{eqnarray*}
  &  &C\bigg(\sum_{\substack{s\in
\bT\backslash\bT_J\\|I_s|>|J|}}\big\|
\big(\ti\Id^{*}_{j,n}\circ\ti\Tr_{\ell,j,m}\big)
 \ti\psi^{*}_{j\ell} \big(f_{\ell,j,0}\circ\ti\Tr_{\ell,j,m}\big)
 \big\|_{L^{p'}(J)} \big\|G_{\ell,j,m}\big\|_p\bigg)^{1/2}|J|\\
 & \leq &C\bigg( \mu \sum_{\substack{s\in
\bT\backslash\bT_J\\|I_s|>|J|}}\frac{|I_s|^{\frac{1}{p}-1}
\big\| (\Id^{**}_{j,n}\circ\ti\Tr_{\ell,j,m} )\psi^{*}_{j\ell}
  (f_{\ell,j,n_\ell}\circ\ti\Tr_{\ell,j,m} )
 \big\|_{p'+1}|J|^{\frac{1}{p'(p'+1)}} }
{\big(1+|I_s|^{-1}{\rm dist}(J,  \ti\Tr^{-1}_{\ell,j,m}(I_s) )\big)^{100}}
 \bigg)^{1/2}|J|\\
 & \leq & C\mu\bigg( \sum_{\substack{s\in
\bT\backslash\bT_J\\|I_s|>|J|}}\frac{|I_s|^{-\frac{1}{p'(p'+1)}}
|J|^{\frac{1}{p'(p'+1)}} } {\big(1+|I_s|^{-1}{\rm dist}(J,
  \ti\Tr^{-1}_{\ell,j,m}(I_s) )\big)^{100}} \bigg)^{1/2}|J|\,\,\leq \,\, C\mu|J|\,.
\end{eqnarray*}
This completes the proof of (\ref{smallBMOp11}).
\end{proof}

\begin{lemma}\label{treeoutp11}
Let $\bT$ be a tree in $\bS$ and ${\bf P}$ be a subset of $\bS$.
Suppose that $\bP\cap\bT=\emptyset$ and  $\bT$ is a maximal tree in $\bP\cup
\bT$. Then we have
\begin{equation}\label{treeout1p11}
\big|\La_{\bP\cup\bT,\Omega,  m}(f_1, f_2, f_3)-\La_{\bP, \Omega, m}(f_1, f_2, f_3)\big|
 \leq \big|\La_{\bT,\Omega,  m}(f_1, f_2, f_3)\big|+Cm\prod_{\ell=1}^2\size^*_\ell(\bT\cup \bP)|I_\bT|\,,
\end{equation}
where $C$ is independent of $f_1, f_2, f_3, L_1, L_2, M_1, M_2, \bP$, $\bT$.
\end{lemma}

The proof is similar to the proof of Lemma \ref{treeoutp11}. We omit
the details and leave it as an exercise to the readers.

\subsection{Proof of Lemma \ref{uniestp11}}

It is easy to prove a size estimate for the trilinear form on a
single tree, that is, for any tree $\bT$,
\begin{equation}\label{sizeestp11}
\big|\La_{\bT, \Omega, m}(f_1, f_2, f_3) \big|\leq Cm^{2/p-1}
\prod_{\ell=1}^2
 \size_\ell^*(\bT)|I_\bT|\,,
\end{equation}
where $C$ is independent of $L_1, L_2, M_1, M_2, m, f_1, f_2, f_3,
\bT$.

In fact, by H\"older inequality, we estimate $|\La_{\bT, \Omega,
m}|$ by
$$
\big\|\ti\Delta_{1,m}^*(\bT)\big\|_{p'}\big\|\Delta_{2,m}^*(\bT)\big\|_{p}\,.
$$
By (\ref{Lqestp11}) and the definition of size, we obtain
(\ref{sizeestp11}) immediately.

Lemma \ref{prilem} still holds for our new sizes of trees and 
$\bS$. Applying this organization
lemma inductively for $\bS$, we decompose
\begin{equation}\label{gooddep11}
\bS=\bigcup_{\sigma} \bS_\si\,,
\end{equation}
where  $\sigma $ ranges over all possible dyadic numbers,
$\bS_\si=\cup_{\bT\in\mathcal F_\si}\bT$ such that $\mathcal F_\si $
is a collection of  maximal trees with
\begin{equation}\label{ctestp11}
 \ct(\bS_\si)\leq C\si^{-p}\,,
\end{equation}
and
\begin{equation}\label{sizekaestp11}
\size^*_\ell(\bS_\si)\leq \si|F_\ell|^{1/p}
\end{equation}
holds for all $\ell\in\{1,2\}$.

By (\ref{Lpest2m}),  the upper bound in the size
estimates for $\bS_\si$ can be sharpened by,
\begin{equation}\label{sizekaest123p11}
\size^*_\ell(\bS_\si)\leq \min\{1, 2^{\beta_\ell m}|F_\ell|^{1/p},
\si|F_\ell|^{1/p} \}\,.
\end{equation}
Hence by Lemma \ref{treeoutp11} and (\ref{sizeestp11}) we estimate
$\La_{\bS,\Omega, m}$ by
$$
 \sum_{\si}\sum_{\bT\in\mathcal F_\si}
 m\prod_{\ell=1}^2\size^*_{\ell}(\bS_\si)|I_\bT|
\,.
$$

Applying (\ref{sizekaest123p11}) and (\ref{ctestp11}), we thus
dominate $\big|\La_{\bS,\Omega, m}(f_1, f_2, f_3)\big|$ by
\begin{equation}\label{finalestpp11}
 Cm\sum_{\sigma }\si^{-p} \min\{1, 2^m|F_1|^{1/p},  \si|F_1|^{1/p}\}
   \min\{1, 2^{\e^2 m}|F_2|^{1/p}, \si|F_2|^{1/p}\}\,,
\end{equation}
which clearly implies (\ref{psmallest11}).
Therefore we complete the proof of Lemma {\ref{uniestp11}}.\\

\end{document}